\documentclass[11pt, reqno]{amsart}
\usepackage{geometry}                 
\usepackage{graphicx}
\usepackage{amssymb,latexsym,amsmath}
\usepackage{amsfonts,amsthm}
\usepackage{amscd}
\usepackage{mathrsfs}
\usepackage{cite,upref}
\usepackage{eucal}
\usepackage{epstopdf}
\usepackage{amsgen}
\usepackage{xspace}
\usepackage{verbatim}

\newcommand{\gd}{\Delta}

\newcommand{\inpt}[1]{\langle #1 \rangle}

\newcommand{\mX}{\mathcal{X}}

\newcommand{\mH}{\mathcal{H}}
\newcommand{\mB}{\mathcal{B}}

\newcommand{\ms}{\mathscr}

\newcommand{\gw}{\Omega}
\newcommand{\ga}{\gamma}
\newcommand{\gb}{\beta}

\newcommand{\gk}{\kappa}
\newcommand{\om}{\omega}

\newcommand{\ve}{\varepsilon}
\newcommand{\snid}{\smallskip \noindent}

\newcommand{\pdr}{\partial}

\newcommand{\csg}{\{ S(t)\}_{t\geq0}}

\newcommand{\tvgm}{\int_{\gw_{M}^v}}
\newcommand{\tzgm}{\int_{\gw_{M}^z}}
\newcommand{\nvgm}[1]{\|#1\|_{\gw_{M}^{|v|}}^2}
\newcommand{\nzgm}[1]{\|#1\|_{\gw_{M}^{|z|}}^2}

\numberwithin{equation}{section}

\theoremstyle{plain}
\newtheorem{theorem}{Theorem}
\newtheorem{lemma}{Lemma}

\theoremstyle{definition}
\newtheorem{definition}{Definition}
\theoremstyle{remark}

\begin{document}

\title[A COUPLED TWO-CELL BRUSSELATOR MODEL]{\textbf{Global Attractor of \\ A Coupled Two-Cell Brusselator Model}}
\dedicatory{Dedicated to George R. Sell on the Occasion of His 70th Birthday}
\author{Yuncheng You} 
\address{Department of Mathematics and Statistics \\
University of South Florida \\
Tampa, FL 33620} 
\email{you@math.usf.edu} 
\date{}

\begin{abstract}
In this work the existence of a global attractor for the solution semiflow of the coupled two-cell Brusselator model equations is proved. A grouping estimation method and a new decomposition approach are introduced to deal with the challenge in proving the absorbing property and the asymptotic compactness of this type of four-variable reaction-diffusion systems with cubic autocatalytic nonlinearity and with linear coupling. It is also proved that the Hausdorff dimension and the fractal dimension of the global attractor are finite. 
\end{abstract}
\maketitle

\section{\textbf{Introduction}}

Consider a coupled two-cell model of reaction-diffusion systems with Brusselator kinetics \cite{GLI07, KS80, aK83, SM82}, 
\begin{align}
	\frac{\pdr u}{\pdr t} &= d_1 \gd u + a - (b + 1)u + u^2 v + D_1 (w - u), \label{equ} \\
	\frac{\pdr v}{\pdr t} &= d_2 \gd v + bu - u^2 v + D_2 (z - v), \label{eqv} \\
	\frac{\pdr w}{\pdr t} &= d_1 \gd w + a - (b + 1)w + w^2 z + D_1 (u - w), \label{eqw} \\
	\frac{\pdr z}{\pdr t} &= d_2 \gd z + bw - w^2 z + D_2 (v - z),  \label{eqz}
\end{align}
for $t > 0$, on a bounded domain $\gw \subset \Re^{n}, n \leq 3$, that has a locally Lipschitz continuous boundary, with the homogeneous Dirichlet boundary condition 
\begin{equation} \label{dbc}
	u(t, x) = v(t, x) = w (t, x) = z (t, x) = 0, \qquad t > 0, \; x \in \partial \gw,
\end{equation}
and an initial condition
\begin{equation} \label{ic}
        u(0,x) = u_0 (x), \; v(0, x) = v_0 (x), \; w(0,x) = w_0 (x), \; z(0, x) = z_0 (x),\qquad x \in \gw,
\end{equation}
where $d_1, d_2, a, b, D_1$, and $D_2$ are positive constants. We do not assume that the initial data $u_0, v_0, w_0, z_0$, nor the solutions $u(t, x), v(t, x), w (t, x), z(t, x)$ to be nonnegative. In this work, we shall study the asymptotic dynamics of the solution semiflow generated by this problem. 

The Brusselator is originally a system of two ordinary differential equations as  a model for cubic autocatalytic chemical or biochemical reactions, cf.  \cite{PL68, AN75, sS94}.  The name is after the hometown of scientists who proposed it.  Brusselator kinetics describes the following scheme of chemical reactions 
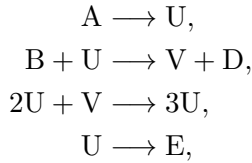
\begin{align*}
	\textup{A} & \longrightarrow \textup{U}, \\
	\textup{B} + \textup{U} & \longrightarrow \textup{V} + \textup{D}, \\
	2 \textup{U} + \textup{V} & \longrightarrow 3 \textup{U}, \\
	\textup{U} & \longrightarrow \textup{E}, 
\end{align*}
where \textup{A}, \textup{B}, \textup{D}, \textup{E}, \textup{U}, and \textup{V} are chemical reactants or products. Let $u(t, x)$ and $v(t, x)$ be the concentrations of \textup{U} and \textup{V}, and assume that the concentrations of the input compounds \textup{A} and \textup{B} are held constant during the reaction process, denoted by $a$ and $b$ respectively. Then one obtains a system of two nonlinear reaction-diffusion equations called (diffusive) \textbf{Brusselator equations},
\begin{align}  
	\frac{\partial u}{\partial t} &= d_1 \gd u + u^{2}v  - (b + 1)u + a,  \label{bru}\\
	\frac{\partial v}{\partial t} &= d_2 \gd v - u^{2}v + bu, \label{brv}
\end{align}

There are several known examples of autocatalysis which can be modeled by the Brusselator equations, such as ferrocyanide-iodate-sulphite reaction, chlorite-iodide-malonic acid reaction, arsenite-iodate reaction, some enzyme catalytic reactions, and fungal mycelia growth, cf. \cite{AO78, AN75, BD95, iE84}.

Numerous studies by numerical simulations or by mathematical analysis, especially after the publications \cite{LMOS93, jP93} in 1993, have shown that the autocatalytic reaction-diffusion systems such as the Brusselator equations and the Gray-Scott equations \cite{GS83, GS84} exhibit rich spatial patterns (including but not restricted to Turing patterns) and complex bifurcations \cite{AO78, BGMLF00, BD95, gD87, DKZ97, FXWO08, IK06, PPG01, PW05, WW03} as well as interesting dynamics \cite{CQ07, iE84, ER83, KCD97, KEW06, yQ07, RR95, RPP97, Wet96, YZE04} on 1D or 2D domains.

For Brusselator equations and the other cubic autocatalytic model equations of space dimension $n \leq 3$, however, we have not seen substantial research results in the front of global dynamics until recently this author proved the existence of a global attractor for Bruuselator equations \cite{yY07}, Gray-Scott equations \cite{yY08}, Selkov equations \cite{yY09a}, and the reversible Schnackenberg equations \cite{yY09b}. 

In this paper, we shall show the existence of a global attractor in the product $L^2$ phase space for the solution semiflow of the coupled two-cell  Brusselator equations \eqref{equ}--\eqref{eqz} with homogeneous Dirichlet boundary conditions \eqref{dbc}.  

This study of global dynamics of the two-cell model of four coupled components is a substantial advance from the one-cell model of two-component reaction-diffusion systems toward the biological network dynamics \cite{GLI07, hK02}. Multi-cell models generically mean the coupled ODEs or PDEs with large number of unknowns (components), which appear widely in the literature of systems biology as well as cell biology. Here understandably "cell" is a generic term that may not be narrowly or directly interpreted as a biological cell. Coupled cells with diffusive reaction and mutual mass exchange are often adopted as model systems for description of processes in living cells and tissues, or in distributed chemical reactions and transport for compartmental reactors \cite{TCN01, SM82}. The mathematical analysis combined with semi-analytical simulations seems to become a common approach to understanding the complicated molecular interactions and signaling pathways in many cases. 

In this regard, unfortunately, the problem with high dimensionality can arise and puzzle the research when the number of molecular species in the system turns out to be very large, which makes the behavior simulation extremely difficult or computationally too expensive. Thus theoretical results on multi-cell model dynamics can give insights to deeper exploration of various signal transductions and tempro-spatial pattern formations. 

For most reaction-diffusion systems consisting of two or more equations arising from the scenarios of autocatalytic chemical reactions or biochemical activator-inhibitor reactions, such as the Brusselator equations and the coupled two-cell Brusselator equations here, the asymptotically dissipative sign condition in vector version
$$
	\lim_{|s| \to \infty} F(s) \cdot s \leq 0
$$
is inherently not satisfied by the opposite-signed and coupled nonlinear terms, see \eqref{opF} later. Besides a serious challenge in dealing with this coupled two-cell model is that, due to the coupling of the two groups of variables $u, v$ and $w, z$, one can no longer make a dissipative \emph{a priori} estimate on the $v$-component by using the $v$-equation separately and then use the sum $y (t, x) = u(t, x) + v (t, x)$ separately to estimate the $u$-component in proving absorbing property and in proving asymptotical compactness of the solution semiflow as we did in \cite{yY07, yY08, yY09a}. The novel mathematical feature in this paper is to overcome this coupling obstacle and make the \emph{a priori} estimates by a method of \emph{grouping estimation} combined with a new decomposition approach.

We start with the formulation of an evolutionary equation associated with the two-cell Brusselator equations. Define the product Hilbert spaces as follows,
\begin{align*}
    	H &= L^2 (\gw) \times L^2 (\gw) \times L^2 (\gw) \times L^2 (\gw), \\
    	E &=  H_{0}^{1}(\gw) \times H_{0}^{1}(\gw) \times H_{0}^{1}(\gw) \times H_{0}^{1}(\gw), \\
    	\Pi =  (H_{0}^{1}(\gw) \cap H^{2}(\gw))& \times (H_{0}^{1}(\gw) \cap H^{2}(\gw)) \times (H_{0}^{1}(\gw) \cap H^{2}(\gw)) \times (H_{0}^{1}(\gw) \cap H^{2}(\gw)).
\end{align*}
The norm and inner-product of $H$ or the component space $L^2 (\gw)$ will be denoted by
$\| \, \cdot \, \|$ and $\inpt{\,\cdot , \cdot\,}$, respectively. The norm of $L^p (\gw)$ will be denoted by $\| \, \cdot \, \|_{L^{p}}$ if $p \ne 2$. By the Poincar$\Acute{e}$ inequality and the homogeneous Dirichlet boundary condition \eqref{dbc}, there is a constant $\ga > 0$ such that
\begin{equation} \label{pcr}
    \| \nabla \varphi \|^2 \geq \ga \| \varphi \|^2, \quad \textup{for} \;  \varphi \in H_{0}^{1}(\gw) \;\; \textup{or} \; \; E,
\end{equation}
and we shall take $\| \nabla \varphi \|$ for the equivalent norm of the space $E$ and of the component space $H_{0}^{1}(\gw)$. We use $| \, \cdot \, |$ to denote an absolute value or a vector norm in a Euclidean space.

It is easy to check that, by the Lumer-Phillips theorem and the analytic semigroup generation theorem \cite{SY02}, the linear operator
\begin{equation} \label{opA}
        A =
        \begin{pmatrix}
            d_1 \gd     & 0    &0    &0\\[3pt]
            0 & d_2 \gd    &0   &0\\[3pt]
            0 &0  &d_1 \gd &0\\[3pt]
            0 &0 &0 &d_2 \gd
        \end{pmatrix}
        : D(A) (= \Pi) \longrightarrow H
\end{equation}
is the generator of an analytic $C_0$-semigroup on the Hilbert space $H$, which will be denoted by $e^{At}, t \geq 0$. By the fact that $H_{0}^{1}(\gw) \hookrightarrow L^6(\gw)$ is a continuous embedding for $n \leq 3$ and using the generalized H\"{o}lder inequality, 
$$
	\| u^{2}v \| \leq \| u \|_{L^6}^2 \| v \|_{L^6}, \quad \| w^{2}z \| \leq \| w \|_{L^6}^2 \| z \|_{L^6}, \quad  \textup{for} \; u, v, w, z \in L^6 (\gw),
$$
one can verify that the nonlinear mapping
\begin{equation} \label{opF}
    F(g) =
        \begin{pmatrix}
            a - (b+1)u + u^2 v + D_1 (w - u)  \\[3pt]
            bu - u^2 v + D_2 (z - v) \\[3pt]
            a - (b+1)w + w^2 z + D_1 (u - w)  \\[3pt]
            bw - w^2 z + D_2 (v - z) 
        \end{pmatrix}
        : E \longrightarrow H,
\end{equation}
where $g = (u, v, w, z)$, is well defined on $E$ and locally Lipschitz continuous. Then the initial-boundary value problem \eqref{equ}--\eqref{ic} is formulated into the following initial value problem,
\begin{equation} \label{eveq} 
    \frac{dg}{dt} = A g + F(g), \quad t > 0, \\[3pt]
\end{equation}
$$
	    g(0) = g_0 = \textup{col} \, (u_0, v_0, w_0, z_0).
$$
where $g (t) = \textup{col} \, (u(t, \cdot), v(t, \cdot), w(t, \cdot), z(t, \cdot))$, simply written as $(u(t, \cdot), v(t, \cdot), w(t, \cdot), z(t, \cdot))$. Accordingly we shall write $g_0 = \textup{col} \, (u_0, v_0, w_0, z_0)$. 

By conducting \emph{a priori} estimates on the Galerkin approximate solutions of the initial value problem \eqref{eveq} and the weak convergence, we can prove the local existence and uniqueness of the weak solution $g(t)$ of \eqref{eveq} in the sense of J. M. Ball specified in \cite{jB77}, which then is shown to be a local mild solution \cite{jB77} and further turns out to be a local strong solution \cite[Theorem 46.2]{SY02}. Moreover, by taking the $H$-inner-product of \eqref{eveq} with this strong solution $g(t)$ itself and conducting \emph{a priori} estimates, one can prove the continuous dependence of the solutions on the initial data and the following property, which is  satisfied by the strong solution $g$,
\begin{equation} \label{soln}
    g \in C([0, T_{max}); H) \cap C^1 ((0, T_{max}); H) \cap L^2 (0, T_{max}; E),
\end{equation}
where $I_{max} = [0, T_{max})$ is the maximal interval of existence.

We refer to \cite{jH88, SY02, rT88} and many references therein for the concepts and basic facts in the theory of infinite dimensional dynamical systems, including few given below for clarity. 

\begin{definition} \label{D:abs}
Let $\{S(t)\}_{t \geq 0}$ be a semiflow on a Banach space $\mX$. A bounded subset $B_0$ of $\mX$ is called an \emph{absorbing set} in $\mX$ if, for any bounded subset $B \subset \mX$, there is some finite time $t_0 \geq 0$ depending on $B$ such that $S(t)B \subset B_0$ for all $t > t_0$.
\end{definition}

\begin{definition} \label{D:atr}
Let $\{S(t)\}_{t \geq 0}$ be a semiflow on a Banach space $\mX$. A subset $\ms{A}$ of $\mX$ is called a \emph{global attractor} for this semiflow, if the following conditions are satisfied: 

(i) $\ms{A}$ is a nonempty, compact, and invariant set in the sense that 
$$
	S(t)\ms{A} = \ms{A} \quad \textup{for any} \; \;  t \geq 0. 
$$

(ii) $\ms{A}$ attracts any bounded set $B$ of $\mX$ in terms of the Hausdorff distance, i.e.
$$
	\text{dist} (S(t)B, \ms{A}) = \sup_{x \in B} \inf_{y \in \ms{A}} \| x - y\|_{\mX} \to 0, \quad \text{as} \; \, t \to \infty.
$$
\end{definition}

\begin{definition} \label{D:asp}
A semiflow $\{S(t)\}_{t \geq 0}$ on a Banach space $\mX$ is called \emph{asymptotically compact} if for any bounded sequences $\{x_n \}$ in $\mX$ and $\{t_n \} \subset (0, \infty)$ with $t_n \to \infty$, there exist subsequences $\{x_{n_k}\}$ of $\{u_n \}$ and $\{t_{n_k}\}$ of $\{t_n\}$, such that $\lim_{k \to \infty} S(t_{n_k})x_{n_k}$ exists in $\mX$.
\end{definition}

Here is the main result of this paper. We emphasize that this result is established unconditionally, neither assuming initial data or solutions are nonnegative, nor imposing any restriction on any positive parameters involved in the equations \eqref{equ}--\eqref{eqz}.

\begin{theorem}[Main Theorem] \label{Mthm}
For any positive parameters $d_1, d_2, a, b, D_1$,and $D_2$, there exists a global attractor $\ms{A}$ in the phase space $H$ for the solution semiflow $\csg$ generated by the coupled two-cell Brusselator evolutionary equation \eqref{eveq}.
\end{theorem}

For investigation of the asymptotic compactness for the two-dell Brusselator semiflow, we shall take the approach of showing the $\gk$-contracting property. Recall the definition of the Kuratowski measure of noncompactness for bounded sets in a Banach space $\mX$,
$$
    \gk (B) \overset{\text{def}}{=} \inf \left\{ \delta : B \; \text{has a finite cover by open sets in $\mX$ of diameters} < \delta \right\}.
$$
If $B$ is an unbounded set, then we define $\gk (B) = \infty$. The basic properties of the Kuratowski measure
are listed here, cf. \cite[Lemma 22.2]{SY02}.

Let $\mX$ be a Banach space and $\gk$ be the Kuratowski measure of noncompactness of bounded sets in $\mX$. Then $\gk$ has the following properties: 

\textup{(i)} $\gk(B) = 0$ if and only if $B$ is precompact in $\mX$, i.e. $Cl_{X} B$ is a compact set in $\mX$. 

\textup{(ii)} $\gk (B_1) \leq \gk (B_2)$ whenever $B_1 \subset B_2$. 

\textup{(iIi)} $\gk (B_1 + B_2) \leq \gk (B_1) + \gk (B_2)$, for any linear sum $B_1 + B_2$. 

The following lemma states concisely the basic result on the existence of a global attractor for a semiflow and provides the connection of the $\gk$-contracting concept to the asymptotical compactness, cf. \cite[Chapter 2]{SY02}.

\begin{lemma} \label{L:kpac}
Let $\csg$ be a semiflow on a Banach space $\mX$. If the following conditions are satisfied\textup{:} 

\textup{(i)} $\csg$ has a bounded absorbing set in $\mX$, and 

\textup{(ii)} $\csg$ is $\gk$-contracting, i.e. $\lim_{t \to \infty} \gk (S(t)B) = 0$ for any bounded set $B \subset \mX$, \\
then $\csg$ is asymptotically compact and there exists a global attractor $\ms{A}$ in $\mX$ for this semiflow. The global attractor is given by
$$
    \ms{A} = \om (B_0) \overset{\textup{def}}{=} \bigcap_{\tau \geq 0} \text{Cl}_{\mX}  \bigcup_{t \geq \tau} (S(t)B_0).
$$.
\end{lemma}
In Section 2 we shall prove the global existence of the strong solutions of the two-cell Brusselator evolutionary equation \eqref{eveq} and the absorbing property of this coupled Brusselator semiflow. In Section 3 a new decomposition technique is presented to deal with the asymptotic compactness issue of this problem. This approach is taken to show the $\gk$-contracting property for the $(v, z)$ components in Section 4 and for the $(u, w)$ components in Section 5, respectively. In Section 6 we assemble these results to prove the existence of a global attractor in the phase space $H$ for the coupled Busselator semiflow and to show that the global attractor has a finite Hausdorff dimension and a finite fractal dimensions. As a remark, with some adjustment in proof, these results are also valid for the homogeneous Neumann boundary condition. Furthermore, corresponding results are valid for the coupled two-cell Gray-Scott equations, Selkov equations, and Schnackenberg equations.

\section{\textbf{Absorbing Property}}

In this paper, we shall write $u(t, x), v(t, x), w(t, x)$, and $z(t,x)$ simply as $u(t), v(t), w(t)$, and $z(t)$, or even as $u, v, w$, and $z$, and similarly for other functions of $(t,x)$. 

\begin{lemma} \label{L:glsn}
For any initial data $g_0 =(u_0, v_0, w_0, z_0) \in H$, there exists a unique, global, strong solution $g(t) = (u(t), v(t), w(t), z(t)), \, t \in [0, \infty)$, of the coupled Brusselator evolutionary equation \eqref{eveq}.
\end{lemma}

\begin{proof}
Taking the inner products $\inpt{\eqref{eqv}, v(t)}$ and $\inpt{\eqref{eqz}, z(t)}$ and summing them up, we get
\begin{equation} \label{vziq}
	\begin{split}
    \frac{1}{2} &\left(\frac{d}{dt} \| v \|^2 + \frac{d}{dt} \| z \|^2 \right) + d_2 \left(\| \nabla v \|^2 + \| \nabla v \|^2\right) \\
    & = \int_{\gw} \left(-u^2 v^2 + buv - w^2 z^2 + bwz - D_2 [v^2 - 2vz + z^2] \right) \, dx   \\
    & = \int_{\gw} - \left[\left(uv - \frac{b}{2}\right)^2 + \left(wz - \frac{b}{2}\right)^2 + D_2 (v - z)^2 \right]\, dx +  \frac{1}{2} b^2 |\gw |\leq \, \frac{1}{2} b^2 |\gw |. 
  \end{split}
\end{equation}
It follows that
\begin{equation*}
    \frac{d}{dt} \left(\| v \|^2 + \| z \|^2 \right) + 2\ga d_2 \left(\| v \|^2 + \| z \|^2 \right) \leq b^2 |\gw |,
\end{equation*}
which yields
\begin{equation} \label{vz}
    \| v(t) \|^2 + \| z (t)\|^2 \leq e^{- 2\ga d_2 t} \left(\| v_0 \|^2 + \| z_0 \|^2 \right) + \frac{b^2 |\gw |}{2\ga d_2}, \quad \text{for} \: t \in [0, T_{max}).
\end{equation}

Let $y(t, x) = u(t, x) + v(t, x) + w(t,x) + z(t,x)$. In order to treat the $u$-component and the $w$-component, first we add up \eqref{equ}, \eqref{eqv}, \eqref{eqw} and \eqref{eqz} altogether to get the following equation satisfied by $y(t) = y(t, x)$, 
\begin{equation} \label{eqy}
    \frac{\pdr y}{\pdr t} = d_1 \gd y - y + \left[(d_2 - d_1)\gd (v + z) + (v + z) + 2a\right].
\end{equation}
Taking the inner-product $\inpt{\eqref{eqy},y(t)}$ we obtain

\begin{align*}
    \frac{1}{2}& \frac{d}{dt} \| y \|^2 + d_1 \| \nabla y \|^2 + \| y \|^2= \int_{\gw} \, \left[(d_2 - d_1)\gd (v+z) + (v+ z) + 2a\right]y \, dx \\[3pt]
    & \leq | d_1 - d_2 | \| \nabla (v + z) \| \| \nabla y \| + \| v + z \| \| y \| + 2a | \gw |^{1/2} \| y \|\\[3pt]
    & \leq \frac{d_1}{2} \| \nabla y \|^2 + \frac{|d_1 - d_2 |^2}{2d_1} \| \nabla (v + z) \|^2 + \frac{1}{2} \| y \|^2 + \| v + z \|^2 +  4a^2 | \gw |,
\end{align*}
so that
\begin{align*}
    \frac{d}{dt} \| y \|^2 + d_1 \| \nabla y \|^2 + \| y \|^2 \leq  \frac{| d_1 - d_2 |^2}{d_1} \| \nabla (v+z) \|^2 + 4\left(\| v\|^2 + \|z \|^2\right) + 8a^2 | \gw |.
\end{align*}
Then we get
\begin{equation} \label{yiq}
    \frac{d}{dt} \| y \|^2 + d_1 \| \nabla y \|^2 + \| y \|^2 \leq \frac{| d_1 - d_2 |^2}{d_1} \| \nabla (v+z) \|^2 + C_1(v_0, z_0, t),
\end{equation}
where
\begin{equation} \label{c1}
    C_1(v_0, z_0, t) = 4 e^{- 2\ga d_2 t} \left(\| v_0 \|^2 + \| z_0 \|^2\right) + \left(\frac{4b^2}{\ga d_2} + 8a^2 \right) |\gw |.
\end{equation}
Integrate the inequality \eqref{yiq} to see that the strong solution $y(t)$ of \eqref{eqy} satisfies the following estimate,
\begin{equation} \label{yy}
	\begin{split}
    \| y(t) \|^2 &\leq \| u_0 + v_0 + w_0 + z_0\|^2 + \frac{| d_1 - d_2 |^2}{d_1} \int_{0}^{t} \| \nabla (v(s)+z(s)) \|^2 \, ds  \\
    {} &+ \frac{2}{\ga d_2} \left(\| v_0 \|^2 + \| v_0 \|^2\right) + \left(\frac{4b^2}{\ga d_2 } + 8a^2 \right) |\gw | \, t, \quad t \in [0, T_{max}).
   \end{split}
\end{equation}
From \eqref{vziq} we have
$$
    d_2 \int_0^{t} \| \nabla (v(s) + z(s)) \|^2 \, ds \leq 2d_2 \int_0^{t} \left(\|\nabla v(s) \|^2 + \|\nabla z(s) \|^2\right)\, ds \leq \left(\| v_0 \|^2 + \|z_0 \|^2\right) + b^2 | \gw | t. 
$$
Substitute this into \eqref{yy} to obtain
\begin{equation} \label{yy1}
	\begin{split}
    \| y(t) \|^2 & \leq \| u_0 + v_0 + w_0 + z_0\|^2 + \left(\frac{| d_1 - d_2 |^2}{d_1 \, d_2} + \frac{2}{\ga d_2} \right) \left( \| v_0 \|^2 + \|z_0 \|^2\right)\\[3pt]
    {} &+ \left[ \left(\frac{| d_1 - d_2 |^2}{d_1 \, d_2} +  \frac{4}{\ga d_2 } \right) b^2+ 8a^2 \right] |\gw | \,t, \quad t \in [0, T_{max}). 
  \end{split}
\end{equation}
Let $p(t) = u(t) + w(t)$. Then by \eqref{vz} and \eqref{yy1} we have shown that
\begin{equation} \label{pp}
	\begin{split}
    \| p(t) \|^2 &= \|u(t) + w(t) \|^2 = \|y(t) - (v(t) + z(t))\|^2 \\
    & \leq 2\left(\|u_0 +v_0+w_0+z_0\|^2 + \|v_0\|^2+\|z_0\|^2\right) + C_2 (g_0) \, t, \quad \textup{for} \; t \in [0, T_{max}),
   \end{split}
\end{equation}
where $C_2 (g_0)$ is a constant depending on the initial data $g_0$.

On the other hand, let $\psi (t,x) = u(t, x) + v(t, x) - w(t,x) - z(t,x)$, which satisfies the equation
\begin{equation} \label{eqps}
    \frac{\pdr \psi}{\pdr t} = d_1 \gd \psi - (1 + 2D_1) \psi + \left[(d_2 - d_1)\gd (v - z) + (1 + 2(D_1 - D_2))(v - z)\right].
\end{equation}
Taking the inner-product $\inpt{\eqref{eqps},\psi(t)}$ we obtain
\begin{align*}
	\frac{1}{2} & \frac{d}{dt} \|\psi \|^2 + d_1 \|\nabla \psi \|^2 + \|\psi \|^2 \leq \frac{1}{2} \frac{d}{dt} \|\psi \|^2 + d_1 \|\nabla \psi \|^2 + (1 + 2D_1) \|\psi \|^2 \\[5pt]
	& \leq (d_1 - d_2) \|\nabla (v - z)\| \|\nabla \psi \| + |1 + 2(D_1 - D_2)| \|v - z\| \|\psi \| \\[3pt]
	& \leq \frac{d_1}{2} \|\nabla \psi\|^2 + \frac{|d_1 - d_2|^2}{2d_1} \|\nabla (v-z)\|^2 + \frac{1}{2} \|\psi\|^2 + \frac{1}{2} |1 + 2(D_1 - D_2)|^2 \|v - z\|^2,
\end{align*}
so that 
\begin{equation} \label{psiq}
	\frac{d}{dt} \|\psi \|^2 + d_1 \|\nabla \psi \|^2 + \|\psi \|^2 \leq \frac{|d_1 - d_2|^2}{d_1} \|\nabla (v-z)\|^2 + C_3 (v_0, z_0, t), 
\end{equation}
where
\begin{equation} \label{c3}
    C_3(v_0, z_0, t) = 2 |1 + 2(D_1 - D_2)|^2 \left( e^{- 2\ga d_2 t} \left(\| v_0 \|^2 + \| z_0 \|^2\right) + \frac{b^2}{2\ga d_2} |\gw|\right).
\end{equation}
Integration of \eqref{psiq} yields
\begin{equation} \label{ppsi}
	\begin{split}
	\|\psi\|^2 & \leq \|u_0+v_0 - w_0 - z_0\|^2 + \frac{|d_1 - d_2|^2}{d_1} \int_0^t \|\nabla (v(s) - z(s) )\|^2\, ds \\
	& + |1 + 2(D_1 - D_2)|^2 \left(\frac{1}{\ga d_2} (\|v_0\|^2 + \|z_0\|^2) + \frac{b^2 |\gw|}{\ga d_2}\, t \right), \quad t \in [0, T_{max}).
	\end{split}
\end{equation}
Note that
$$
    d_2 \int_0^{t} \| \nabla (v(s) - z(s)) \|^2 \, ds \leq 2d_2 \int_0^{t} \left(\|\nabla v(s) \|^2 + \|\nabla z(s) \|^2\right)\, ds \leq \left(\| v_0 \|^2 + \|z_0 \|^2\right) + b^2 | \gw | t. 
$$
From \eqref{ppsi} it follows that
\begin{equation} \label{ppsi1}
	\begin{split}
	\|\psi\|^2 & \leq \| u_0 + v_0 - w_0 - z_0\|^2 + \frac{| d_1 - d_2 |^2}{d_1 \, d_2}  \left( \| v_0 \|^2 + \|z_0 \|^2 + b^2 |\gw|\, t \right) \\[3pt]
    	{} &+  |1 + 2(D_1 - D_2)|^2 \left(\frac{1}{\ga d_2} (\|v_0\|^2 + \|z_0\|^2) + \frac{b^2 |\gw|}{\ga d_2}\, t \right), \quad t \in [0, T_{max}). 
	\end{split}
\end{equation}
Let $q(t) = u(t) - w(t)$. Then by \eqref{vz} and \eqref{ppsi1} we find that
\begin{equation} \label{qq}
	\begin{split}
   	 \| q(t) \|^2 &= \|u(t) - w(t) \|^2 = \|\psi (t) - (v(t) - z(t))\|^2 \\
    	& \leq 2\left(\|u_0 +v_0-w_0-z_0\|^2 + \|v_0\|^2+\|z_0\|^2\right) + C_4 (g_0) \, t, \quad \textup{for} \; t \in [0, T_{max}),
   	\end{split}
\end{equation}
where $C_4 (g_0)$ is a constant depending on the initial data $g_0$.

Finally combining \eqref{pp} and \eqref{qq} we can conclude that for each initial data $g_0 \in H$, both $u(t) = (1/2) (p(t) + q(t))$ and $w(t) = (1/2)(p(t) - q(t))$ components are bounded if $T_{max}$ of the maximal interval of existence of the solution is finite. Together with \eqref{vz}, this shows that, for each $g_0 \in H$, the strong solution $g(t) = (u(t), v(t), w(t), z(t))$ of the equation \eqref{eveq} will never blow up in $H$ at any finite time and it exists globally . 
\end{proof}
Due to Lemma \ref{L:glsn}, the family of all the global strong solutions $\{g(t; g_0), t \geq 0, g_0 \in H \}$ defines a semiflow on $H$,
$$
	S(t): g_0 \mapsto g(t; w_0),  \quad g_0 \in H, \; t \geq 0,
$$
which is called the \textbf{coupled Brusselator semiflow} generated by the coupled Brusselator evolutionary equations.

\begin{lemma} \label{L:absb}
There exists a constant $K_0 > 0$, such that the set 
\begin{equation} \label{bk}
    	B_0 = \left\{ \| g \| \in H : \| g \|^2 \leq K_0 \right\}
\end{equation}
is a bounded absorbing set $B_0$ in $H$ for the coupled Brusselator semiflow $\{S(t)\}_{ t \geq 0}$.
\end{lemma}
\begin{proof}
For this coupled Brusselator semiflow, from \eqref{vz} we obtain
\begin{equation} \label{vzsup}
    	\limsup_{t \to \infty} \, (\| v(t) \|^2 + \|z(t)\|^2) < R_0 = \frac{b^2 |\gw |}{\ga d_2}.
\end{equation}
Moreover, for any $t \geq 0$, \eqref{vziq} also implies that
\begin{equation} \label{vztt}
	\begin{split}
    	\int_{t}^{t+1} (\| \nabla v(s) \|^2 + \|\nabla z(s)\|^2) \, ds  &\leq \frac{1}{d_2} (\| v(t) \|^2 + \|z(t)\|^2 + b^2 |\gw |) \\[3pt]
    	&\leq \frac{1}{d_2} \left( e^{- 2\ga d_2 t} (\| v_0 \|^2 + \|z_0\|^2) + \frac{b^2 |\gw |}{2\ga d_2}\right)
    	+ \frac{b^2 |\gw |}{d_2}. 
	\end{split}
\end{equation}
which is for later use.

From \eqref{yiq} we can deduce that
\begin{equation} \label{eyiq}
    \frac{d}{dt} \left( e^t \| y(t) \|^2 \right) \leq \frac{| d_1 - d_2 |^2}{d_1} \, e^t \| \nabla (v(t) + z(t)) \|^2 + e^t C_1(v_0, z_0, t).
\end{equation}
Integrate \eqref{eyiq} to obtain
\begin{equation} \label{ey}
	\begin{split}
    \| y(t) \|^2 &\leq e^{-t} \| u_0 + v_0 +w_0+z_0 \|^2 \\
    &+ \frac{| d_1 - d_2 |^2}{d_1} \, \int_{0}^{t} e^{- (t -\tau)} \| \nabla (v(\tau) + z(\tau))\|^2 \, d\tau + C_5 (v_0, z_0, t),
    	\end{split}
\end{equation}
where
\begin{align*} 
    C_5 (v_0, z_0, t) &= e^{-t}  \int_{0}^{t}  4e^{(1 - 2\ga d_2)\tau}\, d\tau \, (\| v_0 \|^2 + \|z_0\|^2) + \left( \frac{4b^2}{\ga d_2} + 8a^2 \right) |\gw |  \\
    &\leq 4\alpha (t) (\| v_0 \|^2 + \|z_0\|^2) + \left( \frac{4b^2}{\ga d_2} + 8a^2 \right) |\gw |,
\end{align*}
in which
\begin{equation} \label{apht}
    \alpha (t) = e^{-t}  \int_{0}^{t} e^{(1 - 2\ga d_2)\tau}\, d\tau = 
    \begin{cases}
        \frac{1}{| 1 - 2\ga d_2 |} e^{- 2\ga d_2 t},  & \text{if  $1 - 2\ga d_2 > 0$;} \\[7pt]
        t e^{-t} \leq 2e^{-1} e^{-t/2}, & \text{if $ 1 - 2\ga d_2 = 0$;} \\[7pt]
        \frac{1}{| 1 - 2\ga d_2 | } e^{-t},  &\text{if  $1 - 2\ga d_2 < 0$.}
    \end{cases}
\end{equation}
On the other hand, multiplying \eqref{vziq} by $e^t$ and then integrating each term of the resulting inequality, we get
$$
    \frac{1}{2} \int_{0}^{t} \, e^{\tau} \frac{d}{d\tau} \left(\| v(\tau) \|^2 + \|z(\tau)\|^2 \right) \, d\tau + d_2 \int_{0}^{t} \, e^{\tau} (\| \nabla v(\tau) \|^2 + \|z(\tau)\|^2) \, d\tau \leq \frac{1}{2} b^2 | \gw | e^t,
$$
so that, by integration by parts and using \eqref{vz}, we obtain
\begin{equation} \label{evz}
	\begin{split}
    d_2 &\int_{0}^{t} \, e^{\tau} (\| \nabla v(\tau) \|^2 + \| \nabla v(\tau) \|^2 )\, d\tau  \leq \frac{1}{2} b^2 | \gw | e^t - \frac{1}{2} \int_{0}^{t} \, e^{\tau} \frac{d}{d\tau} \left(\| v(\tau) \|^2 + \| \nabla z(\tau) \|^2 \right) \, d\tau   \\[7pt]
    & = \frac{1}{2} b^2 | \gw | e^t - \frac{1}{2} \left[ e^t (\| v(t) \|^2 + \| z(t) \|^2 ) - (\| v_0 \|^2 + \|z_0\|^2) - \int_{0}^{t} \, e^{\tau} (\| v(\tau) \|^2 + \|z(\tau)\|^2) \, d\tau \right]  \\[7pt]
    & \leq b^2 | \gw | e^t + (\| v_0 \|^2 + \|z_0\|^2) + \int_{0}^{t} \, e^{(1 - 2\ga d_2)\tau} (\| v_0 \|^2 + \|z_0\|^2)\, d\tau + \frac{b^2 |\gw |}{2\ga d_2} e^t   \\[7pt]
    & \leq \left(1 + \frac{1}{2\ga d_2} \right) b^2 |\gw | e^t + \left( 1 + \alpha (t) e^t \right) (\| v_0 \|^2 + \|z_0\|^2), \quad \text{for} \; t \geq 0. 
    	\end{split}
\end{equation}
Substituting \eqref{evz} into \eqref{ey}, we obtain that, for $t \geq 0$,
\begin{equation} \label{yyf}
	\begin{split}
    \| y(t) \|^2 & \leq e^{-t}\| u_0 + v_0 +w_0+z_0\|^2 + C_5 (v_0, z_0, t)   \\[5pt]
     &+ \frac{2| d_1 - d_2 |^2}{d_1\, d_2} e^{-t} \left[ \left(1 + \frac{1}{2\ga d_2} \right) b^2 |\gw | e^t + \left( 1 + e^t \alpha (t) \right) (\| v_0 \|^2+ \|z_0\|^2)\right]   \\
     & \leq e^{-t}\| u_0 + v_0 +w_0 +z_0\|^2 +  4\alpha (t) (\| v_0 \|^2 + \|z_0\|^2) + \left( \frac{4b^2}{\ga d_2} + 8a^2 \right) |\gw |  \\
     &+ \frac{2| d_1 - d_2 |^2}{d_1\, d_2} e^{-t} \left[ \left(1 + \frac{1}{2\ga d_2} \right) b^2 |\gw | e^t + \left( 1 + e^t \alpha (t) \right) (\| v_0 \|^2 + \|z_0\|^2) \right]. 
     	\end{split}
\end{equation}
Note that \eqref{apht} shows $\alpha (t) \rightarrow 0$, as $t \to 0$. From \eqref{yyf} we find that
\begin{equation} \label{ysup}
    \limsup_{t \to \infty} \|y(t) \|^2 < R_1 = 1 + \left( \frac{4b^2}{\ga d_2} + 8a^2 \right) |\gw | +  \frac{2| d_1 - d_2 |^2}{d_1\, d_2}  \left(1 + \frac{1}{2\ga d_2} \right) b^2 |\gw |.
\end{equation}
The combination of \eqref{vzsup} and \eqref{ysup} gives us
\begin{equation} \label{psup}
    \limsup_{t \to \infty} \| u(t) + w(t) \|^2 =  \limsup_{t \to \infty} \| y(t) - (v(t) + z(t))\|^2 < 4R_0 + 2R_1.
\end{equation}

Similarly, from the inequality \eqref{psiq} satisfied by $\psi (t) = u(t) + v(t) - w(t) -z(t)$, we get
\begin{equation} \label{epsiq}
	\frac{d}{dt} \left( e^t \| \psi(t) \|^2 \right) \leq \frac{| d_1 - d_2 |^2}{d_1} \, e^t \| \nabla (v(t) - z(t)) \|^2 + e^t C_3(v_0, z_0, t).
\end{equation}
Integrate \eqref{epsiq} to obtain
\begin{equation} \label{epsi}
	\begin{split}
    \| \psi(t) \|^2 &\leq e^{-t} \| u_0 + v_0 - w_0 - z_0 \|^2 \\
    &+ \frac{| d_1 - d_2 |^2}{d_1} \, \int_{0}^{t} e^{- (t -\tau)} \| \nabla (v(\tau) - z(\tau))\|^2 \, d\tau + C_6 (v_0, z_0, t),
    	\end{split}
\end{equation}
where
\begin{align*} 
    C_6 (v_0, z_0, t) &= 2 |1 + 2(D_1 - D_2)|^2 \left(e^{-t}  \int_{0}^{t} e^{(1 - 2\ga d_2)\tau}\, d\tau \, (\| v_0 \|^2 + \|z_0\|^2) +  \frac{b^2}{\ga d_2} |\gw| \right)   \\
    &\leq 2 |1 + 2(D_1 - D_2)|^2 \left(\alpha (t) (\| v_0 \|^2 + \|z_0\|^2) + \frac{b^2}{\ga d_2} |\gw |\right).
\end{align*}
Using \eqref{evz} to treat the integral term in \eqref{epsi}, we obtain that
\begin{equation} \label{ppsif}
	\begin{split}
    \| \psi(t) \|^2 & \leq e^{-t}\| u_0 + v_0 - w_0 - z_0\|^2 + C_6 (v_0, z_0, t)   \\[2pt]
     &+ \frac{2| d_1 - d_2 |^2}{d_1\, d_2} e^{-t} \int_0^t e^{\tau} (\| \nabla v(\tau) \|^2+ \|\nabla z(\tau)\|^2) \, d\tau  \\
     & \leq e^{-t}\| u_0 + v_0 - w_0 - z_0\|^2 +  2 |1 + 2(D_1 - D_2)|^2 \left(\alpha (t) (\| v_0 \|^2 + \|z_0\|^2) + \frac{b^2}{\ga d_2} |\gw |\right) \\
     &+ \frac{2| d_1 - d_2 |^2}{d_1\, d_2} e^{-t} \left[ \left(1 + \frac{1}{2\ga d_2} \right) b^2 |\gw | e^t + \left( 1 + e^t \alpha (t) \right) (\| v_0 \|^2 + \|z_0\|^2) \right], \quad \textup{for} \; t \geq 0. 
     	\end{split}
\end{equation}
Therefore, since $\alpha (t) \rightarrow 0$, as $t \to 0$, from \eqref{ppsif} we get
\begin{equation} \label{pssup}
    \limsup_{t \to \infty} \| \psi (t) \|^2 < R_2 = 1 + 2b^2 |\gw| \left[\frac{|1 + 2(D_1 - D_2)|^2}{\ga d_2} +  \frac{| d_1 - d_2 |^2}{d_1\, d_2}  \left(1 + \frac{1}{2\ga d_2}\right) \right].
\end{equation}
The combination of \eqref{vzsup} and \eqref{pssup} gives us
\begin{equation} \label{qsup}
    \limsup_{t \to \infty} \| u(t) - w(t) \|^2 =  \limsup_{t \to \infty} \| \psi (t) - (v(t) - z(t))\|^2 < 4R_0 + 2R_2.
\end{equation}

Finally, putting together \eqref{psup} and \eqref{qsup}, we assert that 
\begin{equation} \label{uwsup}
    \limsup_{t \to \infty} (\|u(t)\|^2 + \| w(t) \|^2 ) < 8R_0 + 2(R_1 + R_2).
\end{equation}
Then assembling \eqref{vzsup} and \eqref{uwsup}, we end up with 
\begin{equation*} 
    \limsup_{t \to \infty} \|g(t)\|^2 = \limsup_{t \to \infty} (\|u(t)\|^2 + \|v(t)\|^2 + \| w(t) \|^2 + \|z(t)\|^2 ) < 9R_0 + 2(R_1 + R_2).
\end{equation*}
Thus this lemma is proved with $K_0 =  9R_0 + 2(R_1 + R_2)$ in the absorbing ball $B_0$ in \eqref{bk}. And $K_0$ is a uniform positive constant independent of initial data.
\end{proof}

\section{\textbf{A New Decomposition for Asymptotic Compactness}}

\vspace{3pt}
The lack of inherent dissipativity and the cross-cell coupling make the attempt of showing the asymptotic compactness of the coupled Brusselator semiflow even more challenging. A generic and good idea in dealing with this issue is through a decomposition approach. The existed decomposition methods in different scenarios have been commented in Section 3 of \cite{yY07} and \cite{yY08}.

In \cite{yY07, yY08} the author proved the following lemma, which provides an effective decomposition approach and has been used in proving the $\gk$-contracting property linked to the existence of a global attractor for several cubic autocatalytic reaction-diffusion systems in \cite{yY07, yY08, yY09a, yY09b}.

\begin{lemma} \label{L:odc}
Let $\{\Phi (t)\}_{t \geq 0}$ be the solution semiflow generated by the Brusselator equations \eqref{bru}--\eqref{brv} on $\mH = L^2 (\gw) \times L^2 (\gw)$. Then there exists a global attractor $\ms{A}$ in $\mH$ for this semiflow if and only if the following two conditions are satisfied \textup{:} 

\textup{(i)} There exists a bounded absorbing set $\mB_0$ in $\mH$ for this semiflow. 

\textup{(ii)} For any $\ve > 0$, there are positive constants $M= M(\ve)$ and $T = T(\ve)$ such that
\begin{equation} \label{asym2}
    \int_{\gw(|v(t)| \geq M)} |S(t)w_0|^2 \, dx < C \ve, \quad \text{for any} \; \; t > T, \; w_0 \in \mB_0,  \\
\end{equation}
where $C > 0$ is a uniform constant, and
\begin{equation} \label{asym3}
    \gk \left( (S(t)\mB_0)_{\gw(|v(t)| < M)} \right) \longrightarrow 0, \; \; \text{as} \; \, t \to \infty,
\end{equation}
where
$$
    (S(t)\mB_0)_{\gw(|v(t)| < M)} \overset{\textup{def}}{=} \left\{ (S(t)w_0)(\cdot)\zeta_{M}(\cdot \, ; t, w_0): \textup{for} \; w_0 \in \mB_0 \right\},
$$
in which $\zeta_{M}(x; t, w_0)$ is the characteristic function of the subset $\gw(|v(t)| < M)$, and $v(t) = v(t, x, w_0)$ is the $v$-component of the solution of the Brusselator equations \eqref{bru}--\eqref{brv}.
\end{lemma}

However, this lemma does not work for the investigation of the $\gk$-contracting property with regard to the coupled Brusselator system \eqref{equ}--\eqref{eqz}.

Motivated by Lemma \ref{L:odc}, we now prove a new decomposition technique in the next lemma, which relaxes the decomposing criterion depending on the truncation of one component function (say, the $v$-component) to allowing the criteria depending on the truncation of two component functions (say, the $v$-component and the $z$-component), one for each cell, and the decomposition of the two subgroups of components can be different. But unlike Lemma \ref{L:odc}, the next lemma is only a sufficient condition for the existence of a global attractor. 

\begin{lemma} \label{L:ndc}
	For the solution semiflow $\csg$ generated by the coupled Brusselator evolutionary equation \eqref{eveq} on $H$, there exists a global attractor $\ms{A}$ in $H$, if the following two conditions are satisfied:
	
	\textup{(i)} There exists a bounded absorbing set $B_0$ in $H$ for this semiflow. 

	\textup{(ii)} For any $\ve > 0$, there are positive constants $M= M(\ve)$ and $T = T(\ve)$ such that
\begin{equation} \label{vzm}
    \int_{\gw(|v(t)| \geq M)} |v(t)|^2 \, dx + \int_{\gw(|z(t)| \geq M)} |z(t)|^2 \, dx< L_1 \, \ve, \quad \text{for any} \; \; t > T, \; g_0 \in B_0,  \\
\end{equation}
and
\begin{equation} \label{uwm}
    \int_{\gw(|v(t)| \geq M)} |u(t)|^2 \, dx + \int_{\gw(|z(t)| \geq M)} |w(t)|^2 \, dx< L_2 \, \ve, \quad \text{for any} \; \; t > T, \; g_0 \in B_0,  \\
\end{equation}
where $L_1 > 0$ and $L_2 > 0$ are two uniform positive constants.

	\textup{(iii)} For any given $M > 0$, 
\begin{equation} \label{uvk}
    	\gk \left( P_{u, v} \left[(S(t)B_0)_{\gw(|v(t)| < M)} \right] \right) \longrightarrow 0, \; \; \text{as} \; \, t \to \infty,
\end{equation}
and
\begin{equation} \label{wzk}
    	\gk \left( P_{w, z} \left[(S(t)B_0)_{\gw(|z(t)| < M)} \right] \right) \longrightarrow 0, \; \; \text{as} \; \, t \to \infty,
\end{equation}
where $P_{v,z}$ and $P_{u,w}$ are respectively the orthogonal projections from $H$ onto the component spaces $L^2 (\gw)_v \times L^2 (\gw)_z$ and $L^2 (\gw)_u \times L^2 (\gw)_w$, 
\begin{align}
    (S(t)B_0)_{\gw(|v(t)| < M)} &\overset{\textup{def}}{=} \left\{ (S(t)g_0)(\cdot)\theta_{M}(\cdot \, ; t, g_0): \textup{for} \; g_0 \in B_0 \right\}, \label{tem} \\
    (S(t)B_0)_{\gw(|z(t)| < M)} &\overset{\textup{def}}{=} \left\{ (S(t)g_0)(\cdot)\xi_{M}(\cdot \, ; t, g_0): \textup{for} \; g_0 \in B_0 \right\}, \label{xim}
\end{align}
in which $\theta_{M}(x; t, g_0)$ and $\xi_{M}(x; t, g_0)$ are respectively the characteristic functions of the subsets $\gw(|v(t)| < M)$ and $\gw(|z(t)| < M)$, and $v(t) = v(t, x, g_0)$ is the $v$-component,  $z(t) = z(t, x, g_0)$ is the $z$-component of the solutions of the coupled Brusselator evolutionary equations \eqref{eveq}.
\end{lemma}

\begin{proof}
	In light of Lemma \ref{L:kpac}, it suffices to show that this solution semiflow $\csg$ is $\gk$-contracting on the space $H$. Since the absorbing set $B_0$ in \eqref{bk} attracts every bounded set $B \subset H$, we need only to show 
\begin{equation} \label{kpb}
	\lim_{t \to \infty} \gk (S(t) B_0) = 0.
\end{equation}
By the linear sum property of the Kuratowski measure listed in Section 1, we have
\begin{equation} \label{lsk}
	\gk (S(t)B_0) \leq \gk (S_1 (t) B_0) + \gk (S_2 (t)B_0), \quad t > 0,
\end{equation}
where
$$
	S_1 (t)B_0 = 
	\begin{pmatrix}
            (S(t)B_0)_{u,v}  \\[3pt]
            0
        \end{pmatrix}
	\quad \textup{and} \quad
        S_2 (t)B_0 = 
	\begin{pmatrix}
	   0	\\[3pt]
            (S(t)B_0)_{w,z}  
        \end{pmatrix}, \\[3pt]
$$
and it holds that $\gk (S_1 (t)B_0) = \gk (P_{u,v} [S(t)B_0])$ and $\gk (S_2 (t)B_0) = \gk (P_{w,z} [S(t)B_0])$.

Note that for any given constant $M > 0$, we have
\begin{equation} \label{mnt}
	S_1 (t)B_0 \subset (S_1 (t)B_0)\theta_M + (S_1 (t)B_0)(1 - \theta_M),
\end{equation}
where
\begin{align*}
	(S_1 (t)B_0) \theta_M &= \{(S(t)g_0)(\cdot) \theta_M(\cdot; t, g_0) :  g_0 \in B_0\}, \\
	(S_1 (t)B_0) (1 - \theta_M) &= \{(S(t)g_0)(\cdot) (1 - \theta_M(\cdot; t, g_0)) : g_0 \in B_0\}.
\end{align*}
By \eqref{vzm} and \eqref{uwm}, for an arbitrarily given $\ve > 0$, there exist constants $M > 0$ and $T > 0$ such that
\begin{align*}
	\int_\gw \, | (S_1 (t) & g_0)(x) (1 - \theta_M (x; t, g_0))|^2 \, dx  = \int_{\gw(|v(t)| \geq M)} |S_1 (t) g_0 |^2 \, dx \\
	& = \int_{\gw(|v(t)| \geq M)} \left(|u(t)|^2 + |v(t)|^2 \right)\, dx < L \ve, \quad t > T,
\end{align*}
where $L = L_1 + L_2$, which implies that
\begin{equation} \label{ks1}
	\gk ((S_1 (t)B_0)(1 - \theta_M) < 2 \sqrt{L \ve}, \quad  t > T.
\end{equation}
On the other hand, by \eqref{uvk}, for the same $\ve$ and $M$, there exists a sufficiently large $T^1 > 0$, such that
\begin{equation} \label{ks2}
	\gk ((S_1 (t)B_0)\theta_M) = \gk \left( P_{u, v} \left[(S(t)B_0)_{\gw(|v(t)| < M)} \right] \right) < \ve, \quad t > T^1.
\end{equation}
Then by \eqref{mnt} and the monotone property of the $\gk$-measure, \eqref{ks1} and \eqref{ks2} show that
$$
	\gk (S_1 (t)B_0)) \leq \gk \left((S_1 (t)B_0)\theta_M\right) + \gk \left((S_1 (t)B_0)(1 - \theta_M)\right) < \ve + 2 \sqrt{L \ve}, \; \textup{for}\, t > \max \{T, T^1\}.
$$

Similarly from \eqref{vzm}, \eqref{uwm}, and \eqref{wzk} we can get a sufficiently large $T^2 > 0$, such that
$$
	\gk (S_2 (t)B_0)) \leq \gk \left((S_2 (t)B_0)\xi_M\right) + \gk \left((S_1 (t)B_0)(1 - \xi_M)\right) < \ve + 2 \sqrt{L \ve}, \; \textup{for}\, t > \max \{T, T^2\}.
$$
Finally we substitute the last two inequalities into \eqref{lsk} to conclude that \eqref{kpb} is valid. 
\end{proof}
In the next two sections we shall check the conditions specified in the items (ii) and (iii) of Lemma \ref{L:ndc} toward the proof of the existence of a global attractor for the coupled Brusselator semiflow.

\section{\textbf{$\gk$-Contracting Property for the $(v, z)$ Components}}

\vspace{3pt}

In this section, we shall check that the conditions specified in the items (ii) and (iii) of Lemma \ref{L:ndc} for the $(v, z)$ components of the coupled Brusselator equations. 

In this section and next section, we shall use the notation 
\begin{equation} \label{ntn}
	\begin{split}
	\gw_M^\phi & = \gw(\phi (t) \geq M) = \{x \in \gw : \phi(t, x) \geq M \} \\
	\gw_{|\phi|,M} & = \gw(|\phi (t)| < M) = \{x \in \gw : |\phi(t, x)| < M \}
	\end{split}
\end{equation}
where $\phi (t, x)$ is any measurable function on $\gw$ for each given $t \geq 0$. If a function $\rho (x), x \in \gw$, is in $H$ or $L^2 (\gw)$, then we shall use the following norm notation 
$$
	\| \rho \|_{\gw_M^\phi}^2 = \int_{\gw(\phi (t) \geq M)}  |\rho (x)|^2 \, dx \quad \textup{and} \quad \| \rho \|_{\gw_{|\phi|,M}}^2 = \int_{\gw(|\phi (t)| < M)}  |\rho (x)|^2 \, dx.
$$
We can use $m(S)$ or $| S |$ to denote the Lebesgue measure of a measurable subset $S$ in $\gw$. For any measurable $\phi$ defined on $\gw$, let 
$$
    (\varphi - M)_{+} =
    \begin{cases}
        \varphi (x) - M,    &\textup{if} \; \varphi (x) \geq M, \\[3pt]
        0, &\textup{if} \; \varphi (x) < M;
    \end{cases}
\quad \textup{and} \quad
    (\varphi + M)_{-} =
    \begin{cases}
        \varphi (x) + M,    &\textup{if} \; \varphi (x) \leq -M, \\[3pt]
        0, &\textup{if} \; \varphi (x) > - M.
    \end{cases}
$$

As a preliminary remark, since $B_0$ in Lemma \ref{L:absb} is a bounded absorbing set in $H$ for the coupled Brusselator semiflow $\csg$, there exists a constants $T_0 > 0$, such that
\begin{equation} \label{bkbd}
    \left\| S(t) g_0 \right\|^2 \leq K_0, \quad \text{for any} \; t > T_0, \; g_0 = (u_0, v_0, w_0, z_0) \in B_0,
\end{equation}
where $K_0$ is the constant given in \eqref{bk}. Let this $T_0$ be fixed. 

\begin{lemma} \label{L:vzgm}
For any $\ve > 0$, there exist positive constants $M_1 = M_1(\ve)$ and $T_1 = T_1 (\ve)$, such that the $v$-component $v(t) = v(t, x, g_0)$ and the $z$-component $z(t) = z(t, x, g_0)$ of the solutions of the coupled Brusselator equations \eqref{equ}--\eqref{eqz} satisfy the following estimate,
\begin{equation} \label{vzgm}
    \int_{\gw(|v(t)| \geq M_1)} | v(t) |^2 \, dx + \int_{\gw(|z(t)| \geq M_1)} | z(t) |^2 \, dx < \frac{4b^2}{\ga d_2} \ve, \quad \text{for} \; \; t > T_1, \; g_0 \in B_0,
\end{equation}
where $L_1 \overset{\text{def}}{=} (4b^2)/(\ga d_2)$ is a uniform constant.
\end{lemma}

\begin{proof}
By \eqref{bkbd}, for any $g_0 \in B_0$ and any $t > T_0$, we have $\|v(t)\|^2 + \|z(t)\|^2 \leq K_0$. Hence we have
\begin{align*}
    M^2 &\left[m(\gw (| v(t) | \geq M )) +  m(\gw (| z(t) | \geq M ))\right] \\
    & \leq \int_{\gw(|v(t) |\geq M)} |v(t) |^2 \, dx + \int_{\gw(|z(t) |\geq M)} |v(t) |^2 \, dx \leq K_0,
\end{align*}
so that there exists an $M = M(\ve) > 0$ such that for any $t > T_0, g_0 \in B_0$,
\begin{equation} \label{mle}
    m(\gw (| v(t) | \geq M )) \leq \frac{K_0}{M^2} < \frac{\ve}{2} , \quad \textup{and} \quad m(\gw (| z(t) | \geq M )) \leq \frac{K_0}{M^2} < \frac{\ve}{2} .
\end{equation}
Taking the inner-product $\inpt{\eqref{eqv}, (v(t) - M)_{+}}$, where $M$ is given in \eqref{mle}, we obtain
\begin{equation} \label{vvgm}
	\begin{split}
    \frac{1}{2}& \frac{d}{dt} \| (v - M)_{+} \|^2 + d_2 \tvgm |\nabla (v - M)_{+} |^2 \, dx  \\
    & = - \tvgm u^2 v (v - M)_{+} dx + \tvgm bu (v - M)_{+} dx  + D_2 \tvgm (z - v)(v - M)_+ \, dx   \\
    & \leq - \tvgm \left[ \left(u(v - M)_{+} - \frac{b}{2} \right)^2 + u^2 M (v - M)_{+} \right] \, dx + \frac{b^2}{4} m(\gw(v(t) \geq M)) \\
    & + D_2 \tvgm (z - M)(v - M)_+ \, dx + D_2 M \tvgm (v - M)_+ \, dx \\
    & - D_2 \tvgm  (v - M)_+^2 \,  dx - D_2 M \tvgm (v - M)_+ \, dx \\
    & \leq \frac{b^2}{4} m(\gw(v(t) \geq M)) - D_2 \| (v - M)_+ \|^2 + D_2 \int_{\gw(v(t) \geq M, z(t) \geq M)} (z - M)_+ (v - M)_+ \, dx,
    	\end{split}
\end{equation}
where we noticed that
$$
	 D_2 \int_{\gw(v(t) \geq M, z(t) < M)} (z - M) (v - M)_+ \, dx \leq 0.
$$
Similarly, by taking the inner-product $\inpt{\eqref{eqz}, (z(t) - M)_{+}}$, where $M$ is given in \eqref{mle}, and through parallel steps we can get
\begin{equation} \label{zzgm}
	\begin{split}
    \frac{1}{2} \frac{d}{dt} \| (z - M)_{+} \|^2 &+ d_2 \tzgm |\nabla (z - M)_{+} |^2 \, dx  \leq \frac{b^2}{4} m(\gw(z(t) \geq M)) \\
   &  - D_2 \| (z - M)_+ \|^2 + D_2 \int_{\gw(v(t) \geq M, z(t) \geq M)} (z - M)_+ (v - M)_+ \, dx,
    	\end{split}
\end{equation}
Sum up \eqref{vvgm} and \eqref{zzgm} and then use \eqref{mle} to obtain
\begin{align*}
    \frac{d}{dt} &\left(\| (v - M)_{+} \|^2 + \| (z - M)_{+} \|^2\right)+ 2d_2 \left(\| \nabla (v - M)_+ \|^2 + \| \nabla (v - M)_+ \|^2\right) \leq \frac{b^2}{2} \ve  \\
    & - 2D_2 \left(\| (v - M)_{+}\|^2 - 2 \int_{\gw(v(t) \geq M, z(t) \geq M)} (z - M)_+ (v - M)_+ \, dx\|^2 + \| (z - M)_{+} \|^2\right) \\
    & \leq \frac{b^2}{2} \ve.
\end{align*}
By Poincar\'{e} inequality and Gronwall inequality, it follows that, for $t \geq 0, g_0 \in B_0$,
\begin{equation} \label{vzpm}
    \| (v(t) - M)_{+} \|^2 +  \| (z(t) - M)_{+} \|^2 \leq e^{-2\ga d_2 t} \left(\| (v_0 - M)_{+} \|^2 + \| (z_0 - M)_{+} \|^2\right) + \frac{b^2 \ve}{4 \ga d_2}.
\end{equation} 
Thus there exists a time $T_{+} (\ve) \geq T_0$ such that for any $t > T_{+}$ and any $g_0 \in B_0$, one has
\begin{equation} \label{vzp}
    \| (v(t) - M)_{+} \|^2 + \| (z(t) - M)_{+} \|^2  < \frac{b^2 \ve}{2 \ga d_2}.
\end{equation}

Symmetrically we can prove that there exists a time $T_{-} (\ve) \geq T_0$ such that for any $t > T_{-}$ and any $g_0 \in B_0$, one has
\begin{equation} \label{vzs}
    \| (v(t) + M)_{-} \|^2 + \| (z(t) + M)_{-} \|^2 < \frac{b^2 \ve}{2\ga d_2},
\end{equation}
Adding up \eqref{vzp} and \eqref{vzs}, we find that 
\begin{equation} \label{vzab}
    \int_{\gw(|v(t) | \geq M)} (| v(t)| - M)^2 \, dx + \int_{\gw(|z(t) | \geq M)} (| z(t)| - M)^2 \, dx < \frac{b^2 \ve}{\ga d_2},
\end{equation}
for any $t > T_1 = \max \{T_{+}, T_{-} \}$ and for any $g_0 \in B_0$.

Moreover, since for any $g_0 \in B_0$ and any $T > T_0$, we have
$$
	m (\gw(|v(t)| \geq kM)) + m (\gw(|z(t)| \geq kM)) \leq \frac{K_0}{k^2 M^2},
$$
there exists a sufficiently large integer $k > 0$ such that for any $t > T_1$ and $g_0 \in B_0$, it holds that
\begin{equation} \label{vzd}
	\begin{split}
    \int_{\gw(|v(t)| \geq kM)}& | v(t) |^2 \, dx  +  \int_{\gw(|z(t)| \geq kM)} | z(t) |^2 \, dx \\
    & \leq 2\int_{\gw(|v(t) | \geq M)} (| v(t)| \ - M)^2 \, dx + 2 M^2 m(\gw(|v(t)| \geq kM)) \\ 
    & + 2\int_{\gw(|z(t) | \geq M)} (| z(t)| \ - M)^2 \, dx + 2 M^2 m(\gw(|z(t)| \geq kM)) \\ 
    & \leq \frac{2b^2 \ve}{\ga d_2} + \frac{2M^2 K_0}{k^2 M^2} = \frac{2b^2 \ve}{\ga d_2} + \frac{2K_0}{k^2} < \frac{4b^2 \ve}{\ga d_2}. \\[3pt]
    	\end{split}
\end{equation}
Therefore \eqref{vzgm} is proved with $M_1 = M_1(\ve) = k M$, where $M$ is given in \eqref{mle} and $k$ is the integer that validates \eqref{vzd}, and $T_1 = T_1(\ve) = \max \{T_{+}, T_{-}\}$.
\end{proof}

This lemma shows that the condition \eqref{vzm} in the item (ii) of Lemma \ref{L:ndc} is satisfied for any given $M \geq M_1 (\ve)$ and any $T \geq T_1 (\ve)$, where $M_1$ and $T_1$ are given in Lemma \ref{L:vzgm}.

Let $P_v: H \rightarrow L^2(\gw)_v$ and $P_z: H \rightarrow L^2(\gw)_z$ be the orthogonal projections from $H$ onto the $v$-component space and the $z$-component space, respectively. The next lemma is to check the condition \eqref{uvk} and \eqref{wzk} for the $(v, z)$ components in item (iii) of Lemma \ref{L:ndc}.

\begin{lemma} \label{L:vzlk}
For any given $M > 0$, it holds that 
\begin{align} 
    \gk &\left( P_v (S(t)B_0)_{\gw(|v(t)| < M)} \right) \longrightarrow 0, \; \; \text{as} \;\, t \to \infty, \label{kpv} \\
    \gk &\left( P_z (S(t)B_0)_{\gw(|z(t)| < M)} \right) \longrightarrow 0, \; \; \text{as} \;\, t \to \infty, \label{kpz}
\end{align}
in the space $L^2 (\gw)$, where $(S(t)B_0)_{\gw(|v(t)| < M)}$ and $(S(t)B_0)_{\gw(|z(t)| < M)}$ have been specified in \eqref{tem} and \eqref{xim}.
\end{lemma}

\begin{proof} 
Taking the inner-product $\inpt{\eqref{eqv} , -\gd v(t)}$, we have
$$
    - \inpt{v_t , \gd v} + d_2 \| \gd v \|^2 = \inpt{u^2 v, \gd v} - b\inpt{u, \gd v} - D_2 \inpt{z - v, \gd v}.
$$
By Green's formula and the homogeneous Dirichlet boundary condition, we obtain
$$
    \frac{1}{2} \frac{d}{dt}  \| \nabla v \|^2 + d_2 \| \gd v \|^2 \leq  \inpt{u^2 v, \gd v} + \frac{b^2}{2d_2} \| u \|^2 + \frac{d_2}{2} \| \gd v \|^2 - D_2 \inpt{z, \gd v}  - D_2 \|\nabla v\|^2,
$$
where
\begin{align*}
    \inpt{u^2 v, \gd v} & = - \int_{\gw} \, u^2 | \nabla v |^2 \, dx - 2 \int_{\gw} \, uv (\nabla u \cdot \nabla v) \, dx \\
    & = - \int_{\gw} \, \left| u \nabla v + v \nabla u \right|^2 \, dx + \int_{\gw} \, v^2 | \nabla u |^2 \, dx  \leq \int_{\gw} \, v^2 | \nabla u |^2 \, dx.
\end{align*}
Consequently we get
\begin{equation} \label{nbv}
    \frac{d}{dt} \| \nabla v \|^2 + d_2 \| \gd v \|^2 \leq 2\int_{\gw} \, v^2 | \nabla u |^2 \, dx + \frac{b^2}{d_2} \| u \|^2 - 2D_2 (\|\nabla v \|^2 + \inpt{z, \gd v}).
\end{equation}
Similarly, we can get the following inequality for the $z$-component,
\begin{equation} \label{nbz}
    \frac{d}{dt} \| \nabla z \|^2 + d_2 \| \gd z \|^2 \leq 2\int_{\gw} \, z^2 | \nabla w |^2 \, dx + \frac{b^2}{d_2} \| w \|^2 - 2D_2 (\|\nabla z \|^2 + \inpt{v, \gd z}).
\end{equation}

We can also establish the inequality similar to \eqref{nbv} but with integrals over the set $\gw_{|v|, M} = \gw (| v(t) | < M)$ and the inequality similar to \eqref{nbz} but with integrals over the set $\gw_{|z|, M} = \gw (| z(t) | < M)$. Then sum up the two to obtain 
\begin{equation} \label{nbvz}
	\begin{split}
	\frac{d}{dt} &\left(\|\nabla v \|_{\gw_{|v|, M}}^2 + \|\nabla z \|_{\gw_{|z|, M}}^2\right) + d_2 \left(\|\gd v \|_{\gw_{|v|, M}}^2 + \|\gd z \|_{\gw_{|z|, M}}^2\right) \\[3pt]
	& \leq \frac{b^2}{d_2} \left(\| u \|_{\gw_{|v|, M}}^2 + \| w \|_{\gw_{|z|, M}}^2\right) + 2\int_{\gw_{|v|, M}} v^2 |\nabla u |^2 \, ds + 2 \int_{\gw_{|z|, M}} z^2 |\nabla w |^2 \, dx \\[3pt]
	& - 2D_2 \left(\|\nabla v \|_{\gw_{|v|, M}}^2 + \|\nabla z \|_{\gw_{|z|, M}}^2 + \langle z, \gd v\rangle_{\gw_{|v|, M}} + \langle v, \gd z\rangle_{\gw_{|z|, M}}\right) \\[3pt]
	& \leq \frac{b^2}{d_2} K_0 + 2 M^2 \left( \|\nabla u \|_{\gw_{|v|, M}}^2 + \|\nabla w \|_{\gw_{|z|, M}}^2 \right) \\[3pt]
	& + \frac{2D_2^2}{d_2} \| z \|_{\gw_{|v|, M}}^2 + \frac{d_2}{2} \| \gd v \|_{\gw_{|v|, M}}^2+ \frac{2D_2^2}{d_2} \| v \|_{\gw_{|z|, M}}^2 + \frac{d_2}{2} \| \gd z \|_{\gw_{|z|, M}}^2.
	\end{split}
\end{equation}
Since $\| z \|_{\gw_{|v|, M}}^2 + \| v \|_{\gw_{|z|, M}}^2 \leq \|S(t)g_0 \|^2 \leq K_0$ due to \eqref{bkbd}, and by Poincar\'{e} inequality, it follows that
\begin{equation} \label{nnbvz}
	\begin{split}
	\frac{d}{dt} &\left(\|\nabla v \|_{\gw_{|v|, M}}^2 + \|\nabla z \|_{\gw_{|z|, M}}^2\right) + \frac{\ga d_2}{2} \left(\|\nabla v \|_{\gw_{|v|, M}}^2 + \|\nabla z \|_{\gw_{|z|, M}}^2\right) \\
	& \leq \frac{K_0}{d_2} (b^2 + 2 D_2^2) + 2 M^2 \left( \|\nabla u \|_{\gw_{|v|, M}}^2 + \|\nabla w \|_{\gw_{|z|, M}}^2 \right), \quad t > T_0, \, g_0 \in B_0.
	\end{split}
\end{equation}

This inequality \eqref{nnbvz} implies that
\begin{equation} \label{Grw}
    \frac{d \gb}{dt} \leq r \gb + h, \quad t > T_0,
\end{equation}
where
\begin{align*}
    \gb (t) & = \|\nabla v \|_{\gw_{|v|, M}}^2 + \|\nabla z \|_{\gw_{|z|, M}}^2, \quad r (t)  = \frac{1}{2} \ga d_2, \quad \textup{and} \\
    h(t) & = \frac{K_0}{d_2} (b^2 + 2 D_2^2) + 2 M^2 \left( \|\nabla u \|_{\gw_{|v|, M}}^2 + \|\nabla w \|_{\gw_{|z|, M}}^2 \right).
\end{align*}
By \eqref{vztt}, there exists a constant $T_2 = T_2 (K_0) > 0$ such that $T_2 \geq T_0$ and 
\begin{equation} \label{tbt}
		\begin{split}
    \int_{t}^{t+1} &\left(\| \nabla v(s) \|_{\gw_{|v|, M}}^2 + \| \nabla z(s) \|_{\gw_{|z|, M}}^2 \right) ds \leq \int_{t}^{t+1} \left(\| \nabla v(s) \|^2 + \| \nabla z(s) \|^2\right) ds  \\
    & \leq C_7 = \frac{b^2 | \gw |}{d_2}\left(1 + \frac{1}{\ga d_2}\right), \quad \textup{for} \; t > T_2, \; g_0 \in B_0.
    \end{split}
\end{equation}
By integrating the inequality \eqref{yiq} on the time interval $[t, t + 1]$ and using \eqref{tbt} and \eqref{ysup}, we can deduce that there exists $T_3 = T_3 (K_0) > 0$ such that $T_3 \geq T_2$ and
\begin{equation} \label{tyt}
		\begin{split}
    d_1 \int_{t}^{t+1} \, &\| \nabla y(s) \|^2 \, ds \leq \| y (t) \|^2 + \frac{2|d_1 - d_2 |^2}{d_1} C_7 + 4K_0 + \left( \frac{4b^2}{\ga d_2} + 8a^2 \right) | \gw | \\
    & \leq R_1 + \frac{2|d_1 - d_2 |^2}{d_1} C_7 + 4K_0 + \left( \frac{4b^2}{\ga d_2} + 8a^2 \right) | \gw |, \quad \; t > T_3,
    \end{split}
\end{equation}
where $R_1$ is the constant given in \eqref{ysup}.

Similarly, doing the same to \eqref{psiq} and using \eqref{tbt} and \eqref{pssup}, we find that there exists $T_4 = T_4 (K_0) > 0$ such that $T_4 \geq T_2$ and
\begin{equation} \label{tpst}
		\begin{split}
    d_1 \int_{t}^{t+1} \, &\| \nabla \psi (s) \|^2 \, ds \leq \|\psi (t) \|^2 + \frac{2|d_1 - d_2 |^2}{d_1} C_7 + 2 |1 + 2(D_1 - D_2)|^2 \left(K_0 + \frac{b^2}{2\ga d_2}|\gw|\right) \\
    & \leq R_2 + \frac{2|d_1 - d_2 |^2}{d_1} C_7 + 2 |1 + 2(D_1 - D_2)|^2 \left(K_0 + \frac{b^2}{2\ga d_2}|\gw|\right), \quad \; t > T_4,
    \end{split}
\end{equation}
where $R_2$ is the constant given in \eqref{pssup}.

From \eqref{tbt}, \eqref{tyt} and \eqref{tpst} it follows that, for $t > \max \{T_3, T_4\}$ and any $g_0 \in B_0$, 
\begin{equation} \label{tut}
		\begin{split}
		\int_t^{t+1} \, & \|\nabla u (s)\|_{\gw_{|v|,M}}^2\, ds = \int_t^{t+1} \left\|\frac{1}{2} (y(s) + \psi (s)) - \nabla v(s) \right\|^2\, ds \\
		& \leq \int_t^{t+1} \|\nabla y (s)\|^2\, ds + \int_t^{t+1} \|\nabla \psi (s)\|^2\, ds + 2\int_t^{t+1} \, \|\nabla v (s)\|^2\, ds \leq C_8, 
		\end{split}
\end{equation}
where
\begin{align*}
    C_8 &= 2C_7 \left(1 + \frac{2|d_1 - d_2 |^2}{d_1^2} \right) + \frac{1}{d_1} (R_1 + R_2) + \frac{2K_0}{d_1} \left(2 + |1 + 2(D_1 - D_2)|^2\right) \\
    & + \frac{|\gw|}{d_1} \left[\left(\frac{4b^2}{\ga d_2} + 8a^2 \right) + |1 + 2(D_1 - D_2)^2| \frac{b^2}{\ga d_2}\right].
\end{align*}
We can also assert that
\begin{equation} \label{twt}
		\int_t^{t+1} \, \|\nabla w (s)\|_{\gw_{|z|,M}}^2\, ds = \int_t^{t+1} \left\|\frac{1}{2} (y(s) - \psi (s)) - \nabla z(s) \right\|^2\, ds \leq C_8.
\end{equation}
According to \eqref{tut} and \eqref{twt}, we have 
\begin{equation} \label{tht}
    \int_{t}^{t+1} h(s)\, ds \leq 4 M^2 C_8 + \frac{K_0}{d_2} (b^2 + 2D_2^2),  \quad \textup{for} \; t \geq \max \{T_3, T_4\},\; g_0 \in B_0.
\end{equation}
Besides we have $\int_t^{t+1} r(s)\, ds \leq \ga d_2$.

Finally by \eqref{tbt} and \eqref{tht} and applying the uniform Gronwall inequality \cite{rT88, SY02} to \eqref{Grw}, we obtain 
\begin{equation} \label{cmpt}
    \| \nabla v(t) \|_{\gw_{|v|, M}}^2 + \| \nabla z(t) \|_{\gw_{|z|, M}}^2\leq \left(C_7 + 4 M^2 C_8 + \frac{K_0}{d_2} (b^2 + 2D_2^2)\right) e^{\ga d_2} \quad  t > T_5, \, g_0 \in B_0,
\end{equation}
where $T_5 = \max \{T_3, T_4\} + 1$. Note that the right-hand side of \eqref{cmpt} is a uniform constant depending on the constant $K_0$ in \eqref{bk} and the arbitrarily fixed constant $M$ only. The inequality \eqref{cmpt} shows that for any given $t > T_5$,
\begin{align*}
    P_v &(S(t)B_0)_{\gw(|v(t)| < M)}  \;  \textup{is a bounded set in} \; H_{0}^{1}(\gw),  \, \textup{and} \\
    P_z &(S(t)B_0)_{\gw(|z(t)| < M)}  \;  \textup{is a bounded set in} \; H_{0}^{1}(\gw).
\end{align*}
Due to the compact Sobolev embedding $H_{0}^{1}(\gw) \hookrightarrow L^2 (\gw)$ for space dimension $n \leq 3$, it shows that for any given $t > T_5$,
\begin{align*}
    P_v &(S(t)B_0)_{\gw(|v(t)| < M)}  \;  \textup{is a precompact set in} \; L^2 (\gw), \, \textup{and} \\
    P_z &(S(t)B_0)_{\gw(|z(t)| < M)}  \;  \textup{is a precompact set in} \; L^2 (\gw),
\end{align*}
By the first property of the $\gk$-measure listed in Section 1, \eqref{kpv} and \eqref{kpz} are proved. 
\end{proof}

This lemma shows that the conditions in the item (iii) of Lemma \ref{L:ndc} are verified for the $v$-component in \eqref{uvk} and for the $z$-component in \eqref{wzk}. 

\section{\textbf{$\gk$-Contracting Property for the $(u, w)$ Components}}

\vspace{3pt}

In this section, we shall check that the conditions specified in the items (ii) and (iii) of Lemma \ref{L:ndc} for the $(u, w)$ components of the coupled Brusselator equations.

\begin{lemma} \label{L:uwgm}
For any \,$\ve > 0$, there exist positive constants  $M_2 = M_2(\ve)$ and $T_6 = T_6 (\ve)$ such that the $u$-component $u (t) = u(t, x, g_0)$ and the $w$-component $w(t) = w (t, x, g_0)$ of the solutions of the coupled Brusselator equations \eqref{equ}--\eqref{eqz} satisfy the following estimate,
\begin{equation} \label{uwgm}
    \int_{\gw (|v (t)| \leq M_2)} \, | u (t) |^2 \, dx +  \int_{\gw (|z (t)| \leq M_2)} \, | w (t) |^2 \, dx < L_2 \, \ve, \quad \text{for} \; \, t > T_6, \; g_0 \in B_0,
\end{equation}
where $L^2$ is a uniform positive constant.
\end{lemma}

\begin{proof}
Recall that in our notation $\gw (| v(t) | \geq M)$ and $\gw (| z(t) | \geq M)$ will be denoted by $\gw_M^{| v |}$ and $\gw_M^{| z |}$, respectively. Taking the $L^2$-inner-product $\inpt{\eqref{eqy}, y (t)}$ over the subset $\gw_{M}^{|v|}$, we get
\begin{align*}
    \frac{1}{2} &\frac{d}{dt} \nvgm{y(t)} + d_1\nvgm{\nabla  y(t)} + \nvgm{y(t)}  \\
    & = \int_{\gw_{M}^{|v|}} (d_2 - d_1) y(t) \gd (v(t) + z(t)) \, dx + \int_{\gw_{M}^{|v|}}\, (v(t) + z(t) + 2a)y(t)\,dx \\
    & \leq \frac{d_1}{2} \nvgm{\nabla  y(t)} + \frac{| d_1 - d_2 |^2}{2 d_1} \nvgm{\nabla (v(t) + z(t))} + \frac{1}{2} \nvgm{y(t)} + \frac{1}{2} \nvgm{v(t) + z(t) + 2a},
\end{align*}
so that
\begin{equation} \label{ygd}
	\begin{split}
     \frac{d}{dt} &\nvgm{y(t)} + d_1 \nvgm{\nabla y(t)} + \nvgm{y(t)} \\[3pt]
    & \leq \frac{| d_1 - d_2 |^2}{ d_1} \nvgm{\nabla (v(t) + z(t))} + \nvgm{v(t) + z(t)+ 2a} \\
    & \leq \frac{2| d_1 - d_2 |^2}{ d_1}\left( \nvgm{\nabla v(t)} +  \nvgm{\nabla z(t)} \right) + 4\left(\nvgm{v(t)} + \nvgm{z(t)}\right)+ 8 a^2 |\gw_{M}^{|v|} |. 
    	\end{split}
\end{equation}
Multiply the inequality \eqref{ygd} by $e^t$ and then integrate it on the interval $[0, t]$ to obtain
\begin{equation} \label{ygi}
	\begin{split}
    \nvgm{y(t)} &=  \int_{\gw_{M}^{|v|}} \, | y(t) |^2 \, dx \leq e^{-t}  \nvgm{u_0 + v_0 + w_0 + z_0} \\[3pt]
    & + \frac{2| d_1 - d_2 |^2}{ d_1} e^{-t} \int_{0}^{t} \, e^s \left(\nvgm{\nabla v(s)} + \nvgm{\nabla v(s)}\right)\, ds  \\
    & + e^{-t} \int_{0}^{t} \, e^s \left(4\nvgm{v(s)} + 4\nvgm{z(s)} + 8a^2 |\gw_{M}^{|v|} | \right) \, ds, \quad t \geq 0, \, g_0 \in B_0.  
      	\end{split}
\end{equation}
On the other hand, similar to \eqref{evz}, we have
\begin{equation} \label{vzgi}
	\begin{split}
    d_2 &\int_{0}^{t} \, e^s \left(\nvgm{\nabla v(s)} + \nvgm{\nabla z(s)}\right)\, ds  \\
    & \leq \left(1 + \frac{1}{2\ga d_2} \right) b^2 |\gw_{M}^{|v|} | e^t + \left( 1 + \alpha (t) e^t \right) K_0, \quad t \geq 0, \, g_0 \in B_0. 
    	\end{split}
\end{equation}
Let $\ve > 0$ be arbitrarily given as in Section 4. Since \eqref{mle} implies $|\gw_{M}^{|v|} | < \ve /2$ and $\alpha (t) \to 0$ as $t \to \infty$, and by \eqref{vzgi}, there exists a sufficiently large $\tau_1 = \tau_1 (\ve) \geq T_0$, such that for $t > \tau_1$, the following inequalities hold,
$$
    e^{-t} \nvgm{u_0 + v_0 + w_0 + z_0} \leq 4 K_0 e^{-t} < \ve,
$$
$$
   \frac{2| d_1 - d_2 |^2}{ d_1}  e^{-t} \int_{0}^{t} \, e^s \left( \nvgm{\nabla v(s)} + \nvgm{\nabla z(s)}\right)\, ds < \frac{2| d_1 - d_2 |^2}{ d_1 d_2} \left(1 + \frac{1}{2 d_2 \ga} \right) b^2 \ve,
$$
and
\begin{align*}
	 & e^{-t} \int_{0}^{t} \, e^s \left(4\nvgm{v(s)} + 4\nvgm{z(s)} + 8a^2 |\gw_{M}^{|v|} | \right) \, ds \\
	 & \leq 4 e^{-t} \int_0^t \frac{e^s}{\ga} \left(\nvgm{\nabla v(s)} + \nvgm{\nabla z(s)} \right) ds  +8a^2 |\gw_{M}^{|v|}  <  \frac{4}{\ga d_2} \left(1 + \frac{1}{2 \ga d_2} \right) b^2 \ve + 4 a^2 \ve.
\end{align*}
Substituting these inequalities into \eqref{ygi}, we have
\begin{equation} \label{ygv}
   	\nvgm{y(t)} = \int_{\gw(|v(t)| \geq M)} | y(t) |^2 \, dx  < \Gamma_1 \, \ve, \quad \text{for} \; \; t \geq \tau_1, \; g_0 \in B_0, 
\end{equation}
where $\Gamma_1$ is a uniform constant given by
$$
	\Gamma_1 = 1 + 4a^2 + b^2 \left(1 + \frac{1}{2\ga d_2}\right) \left(\frac{2|d_1 - d_2|^2}{d_1 d_2} + \frac{4}{\ga d_2}\right).
$$
Change $\gw_{M}^{|v|}$ to $\gw_{M}^{|z|}$ and follow the parallel steps as shown above. Then the corresponding inequality is also valid,
\begin{equation} \label{ygz}
   	\nzgm{y(t)} = \int_{\gw(|z(t)| \geq M)} | y(t) |^2 \, dx  < \Gamma_1 \, \ve, \quad \text{for} \; \; t \geq \tau_1, \; g_0 \in B_0. 
\end{equation}

Next taking the inner-product $\inpt{\eqref{eqps}, \psi (t)}$ over the subset $\gw_{M}^{|z|}$, similar to \eqref{ygd} we get
\begin{equation} \label{psgd}
	\begin{split}
     \frac{d}{dt} &\nzgm{\psi (t)} + d_1 \nzgm{\nabla \psi (t)} + \nzgm{\psi (t)} \\[3pt]
    & \leq \frac{| d_1 - d_2 |^2}{ d_1} \nzgm{\nabla (v(t) - z(t))} + |1 + 2(D_1 - D_2)|^2 \nzgm{v(t) - z(t)} \\
    & \leq \frac{2| d_1 - d_2 |^2}{ d_1}\left( \nzgm{\nabla v(t)} +  \nzgm{\nabla z(t)} \right) \\
    & + 2|1 + 2(D_1 - D_2)|^2 \left(\nzgm{v(t)} + \nzgm{z(t)}\right). 
    	\end{split}
\end{equation}
Multiply the inequality \eqref{psgd} by $e^t$ and then integrate it on the interval $[0, t]$. Then by conducting similar estimates we can confirm that, for $t \geq 0$ and $g_0 \in B_0$,
\begin{equation} \label{psgi}
	\begin{split}
    \nzgm{\psi (t)} &=  \int_{\gw_{M}^{|z|}} \, | \psi (t) |^2 \, dx \leq e^{-t}  \nzgm{u_0 + v_0 - w_0 - z_0} \\
    & + \frac{2| d_1 - d_2 |^2}{ d_1} e^{-t} \int_{0}^{t} \, e^s \left(\nzgm{\nabla v(s)} + \nzgm{\nabla v(s)}\right)\, ds  \\
    & + 2 |1 + 2(D_1 - D_2)|^2 e^{-t} \int_{0}^{t} \, e^s \left(\nzgm{v(s)} + \nzgm{z(s)} \right) \, ds.  
      	\end{split}
\end{equation}
Parallel to the argument from \eqref{ygi} through \eqref{ygv}, we can assert that there exists a sufficiently large $\tau_2 = \tau_2 (\ve) \geq T_0$, such that
\begin{equation} \label{psgz}
   	\nzgm{\psi (t)} = \int_{\gw(|z(t)| \geq M)} | \psi (t) |^2 \, dx  < \Gamma_2 \, \ve, \quad \text{for} \; \; t \geq \tau_2, \; g_0 \in B_0, 
\end{equation}
where $\Gamma_2$ is a uniform constant given by
$$
	\Gamma_2 = 1 + 2b^2 \left(1 + \frac{1}{2\ga d_2}\right) \left(\frac{|d_1 - d_2|^2}{d_1 d_2} + \frac{|1 + 2(D_1 - D_2)|^2}{\ga d_2}\right).
$$
Change $\gw_{M}^{|z|}$ to $\gw_{M}^{|v|}$ and follow the parallel steps in proving \eqref{psgz}. Then the corresponding inequality is also valid:
\begin{equation} \label{psgv}
   	\nvgm{\psi (t)} = \int_{\gw(|v(t)| \geq M)} | \psi (t) |^2 \, dx  < \Gamma_2 \, \ve, \quad \text{for} \; \; t \geq \tau_2, \; g_0 \in B_0. 
\end{equation}

Finally, let $M_2 = \max \{M, M_1\}, T_6 = \max \{\tau_1, \tau_2 \}$, where $M$ is given in \eqref{mle} and $M_1$ is the constant in Lemma \ref{L:vzgm} and \eqref{vzgm}. Now we can combine the established \eqref{ygv}, \eqref{psgv}, and the earlier obtained result \eqref{vzgm} to conclude that
\begin{equation} \label{fnu}
	\begin{split}
    	& \int_{\gw(|v(t)| \geq M_2)} | u(t) |^2 \, dx  \leq 2\left(\|u(t) + v(t) \|_{\gw_{M}^{|v|}}^{2}  + \| v(t) \|_{\gw_{M_1}^{|v|}}^{2}\right)  \\[5pt]
    	& = 2 \left\|\frac{1}{2} (y (t) + \psi (t)) \right\|_{\gw_{M}^{|v|}}^{2} + 2\| v(t) \|_{\gw_{M_1}^{|v|}}^{2} \leq \|y (t) \|_{\gw_{M}^{|v|}}^{2} + \|\psi (t) \|_{\gw_{M}^{|v|}}^{2} + 2\| v(t) \|_{\gw_{M_1}^{|v|}}^{2} \\[5pt]
	& < \left(\Gamma_1 + \Gamma_2 + \frac{8b^2}{\ga d_2}\right)\ve, \quad \quad \textup{for} \; t \geq T_6,\;  g_0 \in B_0.
    	\end{split}
\end{equation}
And we can also combine \eqref{ygz}, \eqref{psgz}, and \eqref{vzgm} to conclude that 
\begin{equation} \label{fnw}
	\begin{split}
    	& \int_{\gw(|z(t)| \geq M_2)} | w(t) |^2 \, dx  \leq 2\left(\|w(t) + z(t) \|_{\gw_{M}^{|z|}}^{2}  + \| z(t) \|_{\gw_{M_1}^{|z|}}^{2}\right)  \\[5pt]
    	& = 2 \left\|\frac{1}{2} (y (t) - \psi (t)) \right\|_{\gw_{M}^{|z|}}^{2} + 2\| z(t) \|_{\gw_{M_1}^{|z|}}^{2} \leq \|y (t) \|_{\gw_{M}^{|z|}}^{2} + \|\psi (t) \|_{\gw_{M}^{|z|}}^{2} + 2\| z(t) \|_{\gw_{M_1}^{|z|}}^{2} \\[5pt]
	& < \left(\Gamma_1 + \Gamma_2 + \frac{8b^2}{\ga d_2}\right)\ve, \quad \quad \textup{for} \; t \geq T_6,\;  g_0 \in B_0.
    	\end{split}
\end{equation}
Therefore, \eqref{uwgm} is proved with
$$
	L_2 = 2\left(\Gamma_1 + \Gamma_2 + \frac{8b^2}{\ga d_2}\right).
$$
The proof is complete.
\end{proof}

This lemma shows that the condition \eqref{uwm} in the item (ii) of Lemma \ref{L:ndc} is satisfied for any given $M \geq M_2 (\ve)$ and any $T \geq T_6 (\ve)$, where $M_2$ and $T_6$ are given in Lemma \ref{L:uwgm}.

Let $P_u: H \rightarrow L^2(\gw)_u$ and $P_w: H \rightarrow L^2(\gw)_w$ be the orthogonal projections from $H$ onto the $u$-component space and the $w$-component space, respectively. The next lemma is to check the condition \eqref{uvk} and \eqref{wzk} for the $(u, w)$ components in item (iii) of Lemma \ref{L:ndc}.

\begin{lemma} \label{L:uwlk}
For any given $M > 0$, it holds that 
\begin{align} 
    \gk &\left( P_u (S(t)B_0)_{\gw(|v(t)| < M)} \right) \longrightarrow 0, \; \; \text{as} \;\, t \to \infty, \label{kpu} \\
    \gk &\left( P_w (S(t)B_0)_{\gw(|z(t)| < M)} \right) \longrightarrow 0, \; \; \text{as} \;\, t \to \infty, \label{kpw}
\end{align}
in the space $L^2 (\gw)$, where $(S(t)B_0)_{\gw(|v(t)| < M)}$ and$(S(t)B_0)_{\gw(|z(t)| < M)}$ have been specified in \eqref{tem} and \eqref{xim}.
\end{lemma}
\begin{proof}
Taking the inner-product $\inpt{\eqref{equ},- \gd u(t)}_{\gw_{|v|, M}}$, we can get

\begin{align*}
    \frac{1}{2} &\frac{d}{dt} \| \nabla u \|_{\gw_{|v|,M}}^2 + d_1 \| \gd u \|_{\gw_{|v|,M}}^2 = \int_{\gw_{|v|,M}} u^2 v (-\gd u) \, dx \\
    &+ (b + 1) \int_{\gw_{|v|,M}} u \gd u \, dx - \int_{\gw_{|v|,M}} a \gd u \, dx - D_1 \int_{\gw_{|v|,M}} (w - u) \gd u \, dx\\
    & \leq M \int_{\gw_{|v|,M}} u^2 | \gd u | \, dx + \frac{d_1}{4} \| \gd u \|_{\gw_{|v|,M}}^2 + \frac{1}{d_1} \| (b + 1)u - a \|_{\gw_{|v|,M}}^2 \\
    & + \frac{d_1}{4} \| \gd u \|_{\gw_{|v|,M}}^2 + \frac{D_1^2}{d_1} \| w - u \|_{\gw_{|v|,M}}^2 \\
    &\leq \frac{M^2}{2d_1} \int_{\gw_{|v|,M}} u^4 \, dx + d_1 \| \gd u \|_{\gw_{|v|,M}}^2 + \frac{D_1^2}{d_1} \| w - u \|_{\gw_{|v|,M}}^2 + \frac{1}{d_1} \| (b + 1)u - a \|_{\gw_{|v|,M}}^2.
\end{align*}
Note that for $n \leq 3$ the Sobolev embedding $H_{0}^{1} (\gw) \hookrightarrow L^4 (\gw)$ is continuous and there exists a uniform constant $\delta > 0$ such that
\begin{equation} \label{delta}
    \| \varphi \|_{L^4 (\gw)}^2 \leq \delta \| \varphi \|_{H_{0}^{1}(\gw)}^2, \quad \textup{for} \; \varphi \in H_{0}^{1}(\gw),
\end{equation}
which also holds if $\gw$ is replaced by $\gw_{|v|,M}$. Thus it follows that
\begin{equation} \label{Nbu}
	\begin{split}
    \frac{d}{dt} \| \nabla u \|_{\gw_{|v|,M}}^2 &\leq \frac{M^2}{d_1} \| u \|_{L^4(\gw_{|v|, M})}^4 + \frac{2D_1^2}{d_1} \| w - u \|_{\gw_{|v|,M}}^2 + \frac{4}{d_1} \left( (b + 1)^2 K_0 + a^2 |\gw_{v,M} | \right) \\
    &  \leq \frac{1}{d_1} \left(M^2 \delta^2 \| \nabla u \|^4 + 4K_0 ( (b + 1)^2  + D_1^2) + 4a^2 |\gw | \right) \quad \textup{for} \, t > T_0.
    	\end{split}
\end{equation}
The inequality \eqref{Nbu} can be written as
\begin{equation} \label{Gron}
    \frac{d \gb}{dt} \leq r \gb + h, \quad \textup{for} \; t > T_0,
\end{equation}
in which 
\begin{align*}
    \gb (t) & = \| \nabla u \|_{\gw_{|v|,M}}^2,  \quad r (t) = \frac{M^2 \delta^2}{d_1} \| \nabla u \|_{\gw_{|v|,M}}^2, \quad \textup{and} \\
    h (t) & = \frac{4}{d_1} \left( K_0 ((b + 1)^2 + D_1^2) + a^2 |\gw | \right). 
\end{align*}
In view of \eqref{tut}, we had proved that
$$
	\int_t^{t+1} \|\nabla u (s)\|_{\gw_{|v|,M}}^2\, ds \leq C_8, \quad \textup{for} \; t > T_5 = \max \{T_3 , T_4\} + 1,
$$
where $T_5$ is the same as given in the proof of Lemma \ref{L:vzlk}. Then we can apply the uniform Gronwall inequality to \eqref{Gron} to obtain
\begin{equation} \label{Nbum}
    \| \nabla u(t) \|_{\gw_{|v|, M}}^2 \leq \left( C_8 + \frac{4}{d_1} \left( K_0 ((b + 1)^2 + D_1^2) + a^2 |\gw | \right) \right)\exp \left(\frac{2M^2 \delta^2}{d_1} C_8 \right), 
\end{equation}
for $t > T_5$ and $g_0 \in B_0$. Inequality \eqref{Nbum} shows that 
$$
    P_u (S(t)B_0)_{\gw(|v(t)| < M)}  \;  \textup{is a bounded set in} \; H_{0}^{1}(\gw), \quad \textup{for any} \; t > T_5,
$$
so that 
$$
    P_u (S(t)B_0)_{\gw(|v(t)| < M)}  \;  \textup{is a precompact set in} \; L^2 (\gw),  \quad \textup{for any} \; t > T_5.
$$
Therefore, by the first property of the $\gk$-measure listed in Section 1, \eqref{kpu} is proved. 

Similarly, starting from the inner-product $\inpt{\eqref{eqw},- \gd w(t)}_{\gw_{|z|, M}}$ and using \eqref{twt}, we can also confirm that 
\begin{equation} \label{Nbwm}
    \| \nabla w(t) \|_{\gw_{|z|, M}}^2 \leq \left( C_8 + \frac{4}{d_1} \left( K_0 ((b + 1)^2 + D_1^2) + a^2 |\gw | \right) \right)\exp \left(\frac{2M^2 \delta^2}{d_1} C_8 \right), 
\end{equation}
for $t > T_5$ and $g_0 \in B_0$. Therefore, it holds that
$$
    P_w (S(t)B_0)_{\gw(|z(t)| < M)}  \;  \textup{is a precompact set in} \; L^2 (\gw),  \quad \textup{for any} \; t > T_5.
$$
Consequently \eqref{kpw} is proved. The proof is completed.
\end{proof}

This lemma shows that the conditions in the item (iii) of Lemma \ref{L:ndc} have also been verified for the $u$-component in \eqref{uvk} and for the $w$-component in \eqref{wzk}. 

\section{\textbf{The Existence of a Global Attractor and Its Finite Dimensionality}}

In this section we finally prove Theorem \ref{Mthm} (Main Theorem) on the existence of a global attractor, denoted by $\ms{A}$, for the coupled Brusselator semiflow $\csg$ and that $\ms{A}$ has finite Hausdorff and fractal dimensions.

\begin{proof}[Proof of Theorem \textup{1}]
In Lemma \ref{L:absb}, we have shown that this coupled Brusselator semiflow $\csg$ has a bounded absorbing set $B_0$ in $H$ and the condition (i) in Lemma \ref{L:ndc} is satisfied. 

By Lemma \ref{L:vzgm} and Lemma \ref{L:uwgm}, we have shown that $\csg$ satisfies \eqref{vzm} and \eqref{uwm} with $M = \max \{M_1, M_2\}$ and $T = \max \{T_1, T_6\}$, where $M_1, T_1$ and $M_2, T_6$ are given in Lemma \ref{L:vzgm} and Lemma \ref{L:uwgm} respectively, so that the conditions (ii) in Lemma \ref{L:ndc} is satisfied. 

By Lemma \ref{L:vzlk} and Lemma \ref{L:uwlk}, we proved that \eqref{uvk} and \eqref{wzk} in the condition (iii) of Lemmal \ref{L:ndc} are satisfied by $\csg$. Finally we apply Lemma \ref{L:ndc} to reach the conclusion that there exists a global attractor $\ms{A}$ in $H$ for this coupled Brusselator semiflow $\csg$.
\end{proof}

Let $\ms{A}$ be the global attractor of the semiflow $\csg$ in $H$. Let $q_{m} = \limsup_{t \to \infty} \, q_{m} (t)$, where
\begin{equation} \label{trq} 
	q_{m} (t) = \sup_{g_0 \in \ms{A}} \;  \, \sup_{\substack{g_{i} \in H, \|g_{i} \| = 1\\ i = 1, \cdots, m}} \; \, \left( \frac{1}{t} \int_{0}^{t} \textup{Tr}  \left( A + F^{\prime} (S(\tau) g_0 ) \right) \circ Q_{m} (\tau) \, d\tau \right),
\end{equation}
in which $\textup{Tr} \, (A + F^{\prime} (S(\tau)g_0))$ is the trace of the linear operator $A + F^{\prime} (S(\tau)g_0)$, with $F(g)$ being the nonlinear map in \eqref{eveq}, and $Q_{m} (t)$ stands for the orthogonal projection of the space $H$ on the subspace spanned by $G_1 (t), \cdots, G_{m} (t)$, with
\begin{equation} \label{Frder}
	G_{i} (t) = L(S(t)g_0)g_{i},   \quad i = 1, \cdots, m.
\end{equation}
Here $F^{\prime}(S(\tau)g_0)$ is the Fr\'{e}chet derivative of the map $F$ at  $S(\tau)g_0$, and $L(S(t)g_0)$ is the Fr\'{e}chet derivative of the map $S(t)$ at $g_0$, with $t$ being fixed. 

The following lemma, cf.  \cite[Chapter 5]{rT88},  will be used to show the finite upper bounds of the Hausdorff and fractal dimensions of this global attractor $\ms{A}$.

\begin{lemma} \label{L:HFd}
If there is an integer $m$ such that $q_{m} < 0$, then the Hausdorff dimension $d_{H} (\ms{A})$ and the fractal dimension $d_{F} (\ms{A})$ of $\ms{A}$ satisfy
\begin{equation}  \label{hfd}
	d_{H} (\ms{A}) \leq m,  \quad \textup{and} \quad d_{F} (\ms{A}) \leq m \max_{1 \leq j \leq m - 1} \left( 1 + \frac{(q_{j})_{+}}{| q_{m} |} \right) \leq 2m. 
\end{equation}
\end{lemma} 
It can be shown that for any given $t > 0$, $S(t)$ is Fr\'{e}chet differentiable in $H$ and its Fr\'{e}chet derivative at $g_0$ is the bounded linear operator $L(S(t)g_0)$ given by 
$$
	L(S(t)g_0)G_0 \overset{\textup{def}}{=} G(t) = (U(t), V(t), W(t), Z(t)), \quad \textup{for any} \; G_0 = (U_0, V_0, W_0, Z_0) \in H,
$$
where $(U(t), V(t), W(t), Z(t))$ is the strong solution of the following coupled Brusselator variational equation
\begin{align}
	\frac{\partial U}{\partial t} & = d_1 \gd U + 2u(t)v(t) U + u^2 (t) V - (b + 1) U + D_1 (W - U), \label{varequ} \\
	\frac{\partial V}{\partial t} & = d_2 \gd V - 2u(t)v(t) U - u^2 (t) V + b U + D_2 (Z - V), \label{vareqv} \\
	\frac{\partial W}{\partial t} & = d_1 \gd W + 2w(t)z(t) W + w^2 (t) Z - (b + 1) W + D_1 (U - W), \label{varequ} \\
	\frac{\partial Z}{\partial t} & = d_2 \gd Z - 2w(t)z(t) W - w^2 (t) Z + b W + D_2 (V - Z), \label{vareqv} \\[3pt]
	& U(0) = U_0, \quad V(0) = V_0, \quad W(0) = W_0, \quad Z(0) = Z_0. \label{invv}
\end{align}
Here $g(t) = (u(t), v(t), w(t), z(t)) = S(t)g_0$ is the solution of \eqref{eveq} with the initial condition $g(0) = g_0$.  The initial value problem \eqref{varequ}--\eqref{invv} can be written as 
\begin{equation} \label{vareveq}
	\frac{dG}{dt} = (A + F^{\prime} (S(t)g_0))G,   \quad G(0) = G_0.
\end{equation}

Note that the invariance of $\ms{A}$ implies $\ms{A} \subset B_0$, where $B_0$ is the bounded absorbing set given in Lemma \ref{L:absb}. Hence we have
$$
	\sup_{g_0 \in \ms{A}} \, \| S(t)g_0 \|^2 \leq K_0,
$$
\begin{lemma} \label{L:bdE}
For the global attractor $\ms{A}$ of the coupled Brusselator semiflow $\csg$, there exists a uniform constant $K_1 > 0$ such that 
\begin{equation} \label{klbd}
	\| \nabla g \|^2 \leq K_1, \quad \textup{for any} \; g \in \ms{A}.
\end{equation}
\end{lemma}

\begin{proof}
In \eqref{eveq}, $A: D(A) (= \Pi) \to H$ is a positive sectorial operator and $F \in C_{\textup{loc}}^{\textup{Lip}} (E, H)$. By \eqref{soln}, for any $g_0 \in \ms{A}$, there is a $t_0 \in (0, 1/2)$ such that $S(t_0)g_0 \in E$. By the solution theory in \cite[Section 4.7]{SY02}, one has
\begin{equation} \label{solnpr}
	S(\cdot )g_0 \in C([t_0, \infty), E) \cap C_{\textup{loc}}^{0, \frac{1}{2}} ((t_0, \infty), E) \cap C((t_0, \infty), \Pi),
\end{equation}
were $C_{\textup{loc}}^{0, \frac{1}{2}}$ stands for the space of H\"{o}lder strongly continuous functions with  exponent $1/2$. 

Since $S(t)\ms{A} = \ms{A}$, for any $\widehat{g} \in \ms{A}$ and any $t \geq 1$, there is a particular $g_0 \in \ms{A}$ such that $\widehat{g} = S(t)g_0$. Therefore, the global attractor has the regularity that $\ms{A} \subset E$. 

Next we prove that $\ms{A}$ is a bounded set in $E$, so that \eqref{klbd} holds. Suppose the contrary. Then there exist sequences $\{N_{\ell}\} \subset (0, \infty)$, with $N_{\ell} \geq \ell$, and $\{g_\ell \} \subset \ms{A}$, such that
$$
	\| \nabla g_{\ell} \|\geq N_{\ell}, \quad \ell = 1, 2, \cdots .
$$
Let $g_{\ell}^{0} \in \ms{A}$ be given such that $g_{\ell} = S(N_{\ell})g_{\ell}^{0}$. By \eqref{solnpr}, there is a H\"{o}lder constant $c_{0} > 0$ such that 
$$
	\| \nabla S(t)g_{\ell}^{0} \|^2 \geq \left(N_{\ell} - \frac{c_{0}}{\sqrt{2}} \right)_{+}^2, \quad \textup{for} \; \, t \in I_{\ell} = (N_{\ell} - \frac{1}{2}, N_{\ell} + \frac{1}{2}).
$$
This shows that 
$$
	\int_{N_{\ell} - \frac{1}{2}}^{N_{\ell} + \frac{1}{2}} \, \| \nabla S(\tau)g_{\ell}^{0} \|^2 \, d\tau \geq \left(N_{\ell} - \frac{c_{0}}{\sqrt{2}} \right)_{+}^2 \longrightarrow \infty, \quad \textup{as} \;  \, \ell \to \infty.
$$
This is a contradiction to the fact that for any $g_0 \in \ms{A}$, due to \eqref{vztt},
$$
	\int_{t}^{t + 1} \, (\| \nabla v(s) \|^2 + \| \nabla z(s) \|^2) \, ds \leq \frac{1}{d_2} \left( K_0 + \frac{b^2 |\gw |}{2\ga d_2} \right ) + \frac{b^2 |\gw |}{d_2}, \quad \text{for any } \; t > 0,
$$ 
and, due to \eqref{tut} and \eqref{twt},
$$
	\int_{t}^{t + 1} \, (\| \nabla u(s) \|^2 + \| \nabla w(s) \|^2)\, ds \leq 2C_8, \quad \text{for any} \; t > 0,
$$
where the right-hand side of the above two inequalities are uniform positive constants. Therefore, the conclusion holds.
\end{proof}

\begin{theorem} \label{Dmn}
The global attractors $\ms{A}$ for the coupled Brusselator semiflow $\csg$ has a finite Hausdorff dimesion and a finite fractal dimension.
\end{theorem}

\begin{proof}
By Lemma \ref{L:HFd}, we shall estimate $\textup{Tr} \, (A + F^{\prime} (S(\tau)g_0 )) \circ Q_{m}(\tau)$. At any given time $\tau > 0$, let $\{\varphi_{j} (\tau): j = 1, \cdots , m\}$ be an $H$-orthonormal basis for the subspace 
$$
	Q_m(\tau) H = \textup{Span}\,  \{G_1(\tau), \cdots , G_,(\tau) \},
$$
where $G_1 (t), \cdots , G_{m} (t)$ satisfy \eqref{vareveq} with the respective initial values $G_{1,0}, \cdots , G_{m,0}$ and, without loss of generality, assuming that $G_{1,0}, \cdots , G_{m,0}$ are linearly independent in $H$. By the Gram-Schmidt orthogonalization scheme,  $\varphi_{j}(\tau) = (\varphi_{j}^1 (\tau), \varphi_{j}^2 (\tau),\varphi_{j}^3 (\tau), \varphi_{j}^4 (\tau) ) \in E$, for $j = 1, \cdots , m$,  and $\varphi_{j} (\tau)$ are strongly measurable in $\tau$. Let $d_0 = \min \{d_1, d_2 \}$. Then we have
\begin{equation} \label{Trace}
		\begin{split}
	\textup{Tr} \, (A + F^{\prime} (S(\tau)g_0 )\circ Q_m(\tau)& = \sum_{j=1}^{m} \left( \langle A \varphi_{j}(\tau), \varphi_{j}(\tau) \rangle + \langle  F^{\prime} (S(\tau)g_0 ) \varphi_{j}(\tau), \varphi_{j}(\tau) \rangle\right)  \\
	& \leq - d_0 \sum_{j=1}^{m} \, \| \nabla \varphi_{j}(\tau) \|^2 + J_1 + J_2 + J_3,  
		\end{split}
\end{equation}
where
\begin{align*}
	J_1 & = \sum_{j=1}^{m} \int_{\gw} 2 u(\tau) v(\tau) \left( |\varphi_{j}^1 (\tau) |^2 -  \varphi_{j}^1 (\tau) \varphi_{j}^2 (\tau)  \right) dx  \\
	& + \sum_{j=1}^{m} \, \int_{\gw} 2 w(\tau) z(\tau) \left( |\varphi_{j}^3 (\tau) |^2 -  \varphi_{j}^3 (\tau) \varphi_{j}^4 (\tau)  \right) dx,
\end{align*}
\begin{align*}
	J_2 &= \sum_{j=1}^{m} \int_{\gw} \left( u^2(\tau) \left( \varphi_{j}^1 (\tau) \varphi_{j}^2 (\tau) - | \varphi_{j}^2 (\tau) |^2 \right) + w^2(\tau) \left( \varphi_{j}^3 (\tau) \varphi_{j}^4 (\tau) - | \varphi_{j}^4 (\tau) |^2 \right)\right) dx  \\
	&\leq  \sum_{j=1}^{m} \int_{\gw} \left( u^2(\tau) | \varphi_{j}^1 (\tau) | |\varphi_{j}^2 (\tau)| + w^2(\tau) | \varphi_{j}^3 (\tau) | |\varphi_{j}^4 (\tau)| \right) dx,
\end{align*}
and
\begin{align*}
	J_3 & =  \sum_{j=1}^{m} \int_{\gw} \left( - (b + 1) (|\varphi_{j}^1 (\tau) |^2 + |\varphi_{j}^3 (\tau) |^2) + b (\varphi_{j}^1 (\tau) \varphi_{j}^2 (\tau) + \varphi_{j}^3 (\tau) \varphi_{j}^4 (\tau)) \right) dx \\
	& - \sum_{j=1}^{m} \int_{\gw} \left(D_1 \left(\varphi_j^1 (\tau) - \varphi_j^3 (\tau)\right)^2 + D_2 \left(\varphi_j^3 (\tau) - \varphi_j^4 (\tau)\right)^2 \right) dx \\
	& \leq \sum_{j=1}^{m} \int_{\gw} \, b \left(\varphi_{j}^1 (\tau) \varphi_{j}^2 (\tau) + \varphi_{j}^3 (\tau) \varphi_{j}^4 (\tau) \right) dx.
\end{align*}

We can estimate each of the three terms as follows. First, by the generalized H\"{o}lder inequality and the Sobolev embedding $H_0^1 (\gw) \hookrightarrow  L^4 (\gw)$ for $n \leq 3$, and using Lemma \ref{L:bdE}, we get
\begin{equation} \label{J1eq}
		\begin{split}
	J_1& \leq 2 \sum_{j=1}^{m} \| u(\tau) \|_{L^4} \| v(\tau) \|_{L^4}  \left( \| \varphi_{j}^1 (\tau) \|_{L^4}^2 + \| \varphi_{j}^1(\tau) \|_{L^4}  \| \varphi_{j}^2 (\tau) \|_{L^4} \right) \\
	& + 2 \sum_{j=1}^{m} \| w(\tau) \|_{L^4} \| z(\tau) \|_{L^4}  \left( \| \varphi_{j}^3 (\tau) \|_{L^4}^2 + \| \varphi_{j}^3(\tau) \|_{L^4}  \| \varphi_{j}^4 (\tau) \|_{L^4} \right) \\
	& \leq 4 \sum_{j=1}^{m} \| S(\tau)g_0 \|_{L^4}^2  \| \varphi_{j} (\tau) \|_{L^4}^2 \leq 4 \delta \sum_{j=1}^{m} \| \nabla S(\tau)g_0 \|^2   \| \varphi_{j} (\tau) \|_{L^4}^2 \\
	& \leq 4 \delta K_1  \sum_{j=1}^{m} \| \varphi_{j} (\tau) \|_{L^4}^2, 
		\end{split}
\end{equation}
where $\delta$ is the Sobolev embedding coefficient given in \eqref{delta}. Now we apply the Garliardo-Nirenberg interpolation inequality, cf. \cite[Theorem B.3]{SY02},
\begin{equation} \label{GNineq}
	\| \varphi \|_{W^{k,p}} \leq C \| \varphi \|_{W^{m,q}}^{\theta} \| \varphi \|_{L^{r}}^{1 - \theta}, \quad \textup{for} \; \varphi \in W^{m,q}(\gw),
\end{equation}
provided that $p, q, r \geq 1, 0 < \theta \leq 1$, and
$$
	k - \frac{n}{p} \leq \theta \left( m - \frac{n}{q} \right)  - (1 - \theta ) \frac{n}{r},   \quad \textup{where} \; \, n = \textup{dim} \, \gw.
$$
Here with $W^{k, p}(\gw) = L^4(\gw), W^{m, q}(\gw) = H_{0}^{1}(\gw), L^{r}(\gw) = L^2(\gw)$, and $\theta = n/4 \leq 3/4$, it follows from \eqref{GNineq} that
\begin{equation} \label{inter}
	\| \varphi_{j} (\tau) \|_{L^4} \leq C \| \nabla \varphi_{j} (\tau) \|^{\frac{n}{4}} \| \varphi_{j} (\tau) \|^{1 - \frac{n}{4}} = C \| \nabla \varphi_{j} (\tau) \|^{\frac{n}{4}}, \quad j = 1, \cdots , m,
\end{equation}
since $\| \varphi_{j}(\tau) \| = 1$, where $C$ is a uniform constant. Substitute \eqref{inter} into \eqref{J1eq} to obtain
\begin{equation} \label{J1est}
	J_1 \leq 4\delta K_1 C^2 \sum_{j=1}^{m} \, \| \nabla \varphi_{j} (\tau) \|^{\frac{n}{2}}.
\end{equation}
Similarly, by the generalized H\"{o}lder inequality, we can get
\begin{equation} \label{J2est}
	J_2 \leq \delta K_1 \sum_{j=1}^{m} \| \varphi_{j} (\tau) \|_{L^4}^2 \leq \delta K_1 C^2 \sum_{j=1}^{m} \, \| \nabla \varphi_{j} (\tau) \|^{\frac{n}{2}}.
\end{equation}
Moreover, we have 
\begin{equation} \label{J3est}
	J_3 \leq \sum_{j=1}^{m}  b \| \varphi_{j} (\tau) \|^2 = b m.
\end{equation}
Substituting \eqref{J1est}, \eqref{J2est} and \eqref{J3est} into \eqref{Trace}, we obtain
\begin{equation} \label{Trest}
	\textup{Tr} \, (A + F^{\prime} (S(\tau)g_0 )\circ Q_m(\tau) \leq - d_0 \sum_{j=1}^{m} \| \nabla \varphi_{j}(\tau) \|^2 + 5\delta K_1 C^2 \sum_{j=1}^{m} \| \nabla \varphi_{j}(\tau) \|^{\frac{n}{2}} + b m.
\end{equation}
By Young's inequality, for $n \leq 3$, we have
$$
	5 \delta K_1 C^2 \sum_{j=1}^{m} \| \nabla \varphi_{j}(\tau) \|^{\frac{n}{2}}  \leq \frac{d_0}{2} \sum_{j=1}^{m} \|\nabla \varphi_{j}(\tau) \|^2 + K_2(n) m,
$$
where $K_2(n)$ is a uniform constant depending only on $n = $ dim $(\gw)$. Hence,
\begin{equation*} 
	\textup{Tr} \, (A + F^{\prime} (S(\tau)g_0 )\circ Q_m(\tau) \leq - \frac{d_0}{2} \sum_{j=1}^{m} \| \nabla \varphi_{j}(\tau) \|^2 + \left( K_2 (n) + b \right) m, \quad \tau > 0, \, g_0 \in \ms{A}. 
\end{equation*}
According to the generalized Sobolev-Lieb-Thirring inequality \cite[Appendix, Corollary 4.1]{rT88}, since $\{ \varphi_1 (\tau), \cdots , \varphi_{m} (\tau) \}$ is an orthonormal set in $H$, so there exists a uniform constant $K_3 > 0$ only depending on the shape and dimension of $\gw$, such that
\begin{equation} \label{SLTineq}
	\sum_{j=1}^{m} \| \nabla \varphi_{j}(\tau) \|^2  \geq K_3 \frac{m^{1 + \frac{2}{n}}}{|\gw |^{\frac{2}{n}}}.
\end{equation}
Therefore, we end up with
\begin{equation} \label{finest}
	\textup{Tr} \, (A + F^{\prime} (S(\tau)g_0 )\circ Q_m(\tau) \leq - \frac{d_0 K_3}{2 |\gw |^{\frac{2}{n}}} m^{1 + \frac{2}{n}} + \left( K_2 (n) + b \right) m, \quad \tau > 0,\, g_0 \in \ms{A}.
\end{equation}
Then we can conclude that
\begin{equation} \label{qmt}
		\begin{split}
	q_{m}(t) & =  \sup_{g_0 \in \ms{A}} \;  \, \sup_{\substack{g_{i} \in H, \|g_{i} \| = 1\\ i = 1, \cdots, m}} \; \, \left( \frac{1}{t} \int_{0}^{t} \textup{Tr}  \left( A + F^{\prime} (S(\tau) g_0 ) \right) \circ Q_{m} (\tau) \, d\tau \right)  \\
	& \leq - \, \frac{d_0 K_3}{2 |\gw |^{\frac{2}{n}}} m^{1 + \frac{2}{n}}  + \left(K_2(n) + b\right) m, \quad \textup{for any} \; t > 0, 
		\end{split}
\end{equation}
so that 
\begin{equation}  \label{qm}
	q_m = \limsup_{t \to \infty} \, q_m (t) \leq  - \, \frac{d_0 K_3}{2 |\gw |^{\frac{2}{n}}} m^{1 + \frac{2}{n}}  + \left(K_2(n) + b\right) m < 0,
\end{equation}
if the integer $m$ satisfies the following condition,
\begin{equation} \label{dimc}
	m - 1 \leq \left( \frac{2(K_2(n) + b)}{d_0 K_3} \right)^{n/2} | \gw | < m.
\end{equation}
According to Lemma \ref{L:HFd}, we have shown that the Hausdorff dimension and the fractal dimension of the global attractor $\ms{A}$ are finite and their upper bounds are given by
$$
	d_{H} (\ms{A}) \leq m \quad \textup{and} \quad d_{F} (\ms{A}) \leq 2m,
$$
where $m$ satisfies \eqref{dimc}. 
\end{proof}

\end{document}